\documentclass[11pt]{amsart}  

\usepackage[a4paper,hmargin={2.5cm,2.5cm},vmargin={2cm,2cm}]{geometry}

\usepackage{amsmath, amsthm, amssymb, eucal}
\usepackage{enumerate}
\usepackage{stmaryrd}
\usepackage[all,arc]{xy}
\usepackage{longtable}

\def\sizexy{2em} 

\theoremstyle{plain}
\newtheorem{theorem}{Theorem}[section]
\newtheorem{lemma}[theorem]{Lemma}
\newtheorem{prop}[theorem]{Proposition}

\theoremstyle{definition}
\newtheorem{definition}[theorem]{Definition}

\theoremstyle{remark}

\def\mathrmdef#1{\expandafter\def\csname#1\endcsname{{\rm#1}}}

\mathrmdef{Alg}
\mathrmdef{co}
\mathrmdef{false}
\mathrmdef{Fuz}
\mathrmdef{Id}
\mathrmdef{Mat}
\mathrmdef{ob}
\mathrmdef{op}
\mathrmdef{Rel}
\mathrmdef{Span}
\mathrmdef{true}

\def\mathbfdef#1{\expandafter\def\csname#1\endcsname{{\rm\bf#1}}}
\mathbfdef{App}\mathbfdef{Cat} \mathbfdef{CAT} \mathbfdef{CompHaus} \mathbfdef{D}  \mathbfdef{Mon} \mathbfdef{Nach} \mathbfdef{Ord} \mathbfdef{PrMet}\mathbfdef{MultiCat}
\mathbfdef{PrSet}\mathbfdef{PsTop}\mathbfdef{Rel}\mathbfdef{Set}\mathbfdef{Top}\mathbfdef{UltraCat}\mathbfdef{V}

\newcommand{\calO}{\mathcal{O}}

\newcommand\adjunct[2]{\xymatrix@=8ex{\ar@{}[r]|{\top}\ar@<1mm>@/^2mm/[r]^{{#2}} & \ar@<1mm>@/^2mm/[l]^{{#1}}}}

\newcommand{\relto}{{\longrightarrow\hspace*{-2.8ex}{\mapstochar}\hspace*{2.6ex}}}

\newcommand{\kto}{\relbar\joinrel\rightharpoonup}
\newcommand{\krelto}{\,{\kto\hspace*{-2.5ex}{\mapstochar}\hspace*{2.6ex}}}

\newcommand{\mon}[1]{{\mathbb{#1}}}

\newcommand{\mT}{\mon{T}}

\newcommand{\mU}{\mon{U}}
\newcommand{\mM}{\mon{M}}

\newcommand{\mmonad}{(M,e,m)}

\newcommand{\catfont}[1]{\mathbf{#1}}

\newcommand{\catB}{\catfont{B}}
\newcommand{\ORDCH}{\catfont{OrdCompHaus}}
\newcommand{\TOP}{\catfont{Top}}
\newcommand{\STCOMP}{\catfont{StablyComp}}

\newcommand{\MONCAT}{\catfont{MonCat}}
\newcommand{\SPAN}{\catfont{Span}}
\newcommand{\MULTICAT}{\catfont{MultiCat}}
\newcommand{\REPMULTICAT}{\catfont{RepMultiCat}}

\newcommand{\Rels}[1]{#1\text{-}\catfont{Rel}}
\newcommand{\Cats}[1]{#1\text{-}\catfont{Cat}}

\def\VRel{\V\mbox{-}\Rel}
\def\TVCat{(\mT,\V)\mbox{-}\Cat}
\def\VCat{\V\mbox{-}\Cat}

\def\xx{\mathfrak{x}}
\def\yy{\mathfrak{y}}

\def\ea{\hat{a}}
\def\ebeta{\hat{\beta}}

\def\lrw{\longrightarrow }  
\def\rel{\longrightarrow\hspace*{-2.8ex}{\mapstochar}\hspace*{2.6ex}}
\def\fii{\varphi }
\def\two{\mbox{\sf 2}}

\def\cat{cat\-egory }

\title{Representable $(\mT, \V)$-categories}
\author{D. Chikhladze}
\address{CMUC, Department of Mathematics, University of Coimbra, 3001-501 Coimbra, Portugal}
\email{dimitri@mat.uc.pt}
\author{M. M. Clementino}
\address{CMUC, Department of Mathematics, University of Coimbra, 3001-501 Coimbra, Portugal}
\email{mmc@mat.uc.pt}
\author{D. Hofmann}
\address{CIDMA -- Center for Research and Development in Mathematics and Applications, Department of Mathematics, University of Aveiro, 3810-193 Aveiro, Portugal}
\email{dirk@ua.pt}

\date{}

\begin{document}  

\begin{abstract}  
Working in the framework of $(\mT,\V)$-categories, for a symmetric monoidal closed category $\V$ and a (not necessarily cartesian) monad $\mT$, we present a common account to the study of ordered compact Hausdorff spaces and stably compact spaces on one side and monoidal categories and representable multicategories on the other one. In this setting we introduce the notion of dual for $(\mT,\V)$-categories.
\vspace{0.1cm}  

\hspace*{-\parindent}{\em  Mathematics Subject Classification}: 18C20, 18D15, 18A05, 18B30, 18B35.  
\vspace{.1cm}  

\end{abstract}

\maketitle

\section{Introduction}

The principal objective of this paper is to present a common account to the study of ordered compact Hausdorff spaces and stably compact spaces on one side and monoidal categories and representable multicategories on the other one. Both theories have similar features but were developed independently.

On the topological side, the starting point is the work of Stone on the representation of Boolean algebras \cite{Sto36} and distributive lattices \cite{Sto38}. In the latter paper, Stone proves that (in modern language) the category of distributive lattices and homomorphisms is dually equivalent to the category of spectral topological spaces and spectral maps. Here a topological space is spectral whenever it is sober and the compact open subsets form a basis for the topology which is closed under finite intersections; and a continuous map is called spectral whenever the inverse image of a compact open subset is compact. Later Hochster \cite{Hoc69} showed that spectral spaces are, up to homeomorphism, the prime spectra of commutative rings with unit, and in the same paper he also introduced a notion of dual spectral space. A different perspective on duality theory for distributive lattices was given by Priestley in \cite{Pri70}: the category of distributive lattices and homomorphisms is also dually equivalent to the category of certain ordered compact Hausdorff spaces (introduced by Nachbin in \cite{Nac50}) and continuous monotone maps. In particular, this full subcategory of the category of ordered compact Hausdorff spaces is equivalent to the category of spectral spaces. In fact, this equivalence generalises to all ordered compact Hausdorff spaces: the category $\ORDCH$ of ordered compact Hausdorff spaces and continuous monotone maps is equivalent to the category $\STCOMP$ of stably compact spaces and spectral maps (see \cite{GHK+80}). Furthermore, as shown in \cite{Sim82} (see also \cite{EF99}), stably compact spaces can be recognised among all topological spaces by a universal property; namely, as the algebras for a Kock-Z\"oberlein monad (or lax idempotent monad, or simply KZ; see \cite{Ko}) on $\TOP$. Finally, Flagg \cite{Fla97a} proved that $\ORDCH$ is also monadic over ordered sets.

Independently, a very similar scenario was developed by Hermida in \cite{Her00,Her01} in the context of higher-dimensional category theory, now with monoidal categories and multicategories in lieu of ordered compact Hausdorff spaces and topological spaces. More specifically, he introduced in \cite{Her00} the notion of representable multicategory and constructed a 2-equivalence between the 2-category of representable multicategories and the 2-category of monoidal categories; that is, representable multicategories can be seen as a higher-dimensional counterpart of stably compact topological spaces. More in detail, we have the following analogies:
\begin{longtable}{rl}
 topological space & multicategory,\\
 ordered compact Hausdorff space & monoidal category,\\
 stably compact space & representable multicategory;
\end{longtable}
\noindent and there are KZ-monadic 2-adjunctions
\begin{align*}
 \ORDCH\adjunct{}{}\Top && \MONCAT\adjunct{}{}\MULTICAT,
\intertext{which restrict to 2-equivalences}
 \ORDCH\simeq\STCOMP && \MONCAT\simeq\REPMULTICAT.
\end{align*}

To bring both theories under one roof, we consider here the setting used in \cite{CT03} to introduce $(\mT,\V)$-categories; that is, a symmetric monoidal closed category $\V$ together with a (not necessarily cartesian) monad $\mT$ on $\Set$ laxly extended to the bicategory $\Rels{\V}$ of $\V$-relations. After recalling the notions of $(\mT,\V)$-categories and $(\mT,\V)$-functors, we proceed showing that the above-mentioned results hold in this setting: the $\Set$-monad $\mT$ extends naturally to $\Cats{\V}$, and its Eilenberg--Moore category admits an adjunction
\[
 (\Cats{\V})^\mT\adjunct{}{}\Cats{(\mT,\V)},
\]
so that the induced monad is of Kock-Z\"oberlein type. Following the terminology of \cite{Her00}, we call the pseudo-algebras for the induced monad on $\Cats{(\mT,\V)}$ representable $(\mT,\V)$-categories. We explain in more detail how this notion captures both theories mentioned above. Finally, we introduce a notion of dual $(\mT,\V)$-category. We recall that this concept turned out to be crucial in the development of a completeness theory for $(\mT,\V)$-categories when $\V$ is a quantale, i.e.\ a small symmetric monoidal closed complete category (see \cite{CH09}).
\smallskip

From a more formal point of view, $(\mT, \V)$-categories are monads within a certain bicategory-like structure. Some of the theory presented in this paper is ``formal monad theoretical" in character. This perspective will be developed in an upcoming paper \cite{Ch14}.

\section{Basic assumptions}

Throughout the paper \emph{$\V$ is a complete, cocomplete, symmetric
monoidal-closed category, with tensor product $\otimes$ and unit
$I$}. Normally we avoid explicit reference to the natural unit,
associativity and symmetry isomorphisms.

The bi\cat $\VRel$ of $\V$-relations (also  called $\Mat(\V)$: see \cite{BCSW, RW}) has as \vspace*{1mm}\\
-- objects sets, denoted by $X$, $Y$, $\dots$, also
considered as (small) discrete categories,\vspace*{1mm}\\
-- arrows (=1-cells) $r:X\rel Y$ are families of $\V$-objects
$r(x,y)$ ($x\in X,\,y\in Y$),\vspace*{1mm}\\
-- 2-cells $\varphi:r\to r'$ are families of morphisms
$\varphi_{x,y}:r(x,y)\to r'(x,y)$ ($x\in X,\,y\in Y$) in $\V$,
i.e., natural transformations $\fii:r\to r'$; hence, their
(vertical) composition is computed componentwise in $\V$:
\[(\fii'\cdot \fii)_{x,y}=\fii'_{x,y}\fii_{x,y}.\]
The (horizontal) composition of arrows $r:X\rel Y$ and $s:Y\rel Z$
is given by {\em relational multiplication}:
\[(sr)(x,z)=\sum_{y\in Y}\;r(x,y)\otimes s(y,z),\]
which is extended naturally to 2-cells; that is,  for $\fii:r\to
r'$, $\psi:s\to s'$,
\[(\psi\fii)_{x,z}=\sum_{y\in
Y}\;\fii_{x,y}\otimes\psi_{y,z}:(sr)(x,z)\to(s'r')(x,z).\]

There is a pseudofunctor $\Set\lrw\VRel$ which maps objects
identically and treats a $\Set$-map $f:X\to Y$ as a $\V$-relation
$f:X\relto Y$ in $\VRel$, with $f(x,y)=I$ if $f(x)=y$ and
$f(x,y)=\bot$ otherwise, where $\bot$ is a fixed initial object of
$\V$. If an arrow $r:X\relto Y$ is given by a $\Set$-map, we shall
indicate this by writing $r:X\to Y$, and by normally using
$f,\,g,\,\dots$, rather than $r,\,s,\,\dots$.

Like for $\V$, in order to simplify formulae and diagrams, we disregard the unity and associativity isomorphisms
in the  bicategory $\VRel$ when convenient.

$\VRel$ has a pseudo-involution, given by {\em
transposition}: the transpose $r^\circ :Y\rel X$ of $r:X\rel Y$ is
defined by $r^\circ (y,x)=r(x,y)$; likewise for 2-cells. In
particular, there are natural and coherent isomorphisms
\[(sr)^\circ \cong r^\circ s^\circ \]
involving the symmetry isomorphisms of $\V$. The
transpose $f^\circ $ of a $\Set$-map $f:X\to Y$ is a
right adjoint to $f$ in the bi\cat $\VRel$, so that $f$ is really a
``map" in Lawvere's sense; hence, there are 2-cells
\[\xymatrix{1_X\ar[r]^{\lambda_f}& f^\circ f&\mbox{ and }&ff^\circ \ar[r]^{\rho_f}& 1_Y}\]satisfying the triangular
identities.
\bigskip

We fix a \emph{monad
$\mT=(T,e,m)$ on $\Set$ with a
lax extension to $\VRel$}, again denoted by $\mT$, so that:
\begin{enumerate}[--]
\item There is a lax functor $T:\VRel\to\VRel$ which extends
the given $\Set$-functor; hence, for an arrow $r:X\rel Y$ we are
given $Tr:TX\rel TY$, with $Tr$ a $\Set$-map if $r$ is one,
and $T$ extends to 2-cells functorially:
\[T(\fii'\cdot\fii)=T\fii'\cdot T\fii,\;\;T1_r=1_{Tr};\]
furthermore, for all $r:X\rel Y$ and $s:Y\rel Z$ there are natural and
coherent 2-cells
\[\kappa=\kappa_{s,r}:TsTr\lrw T(sr),\]
so that the following diagrams commute:
\begin{equation}\tag{lax}\label{eq:lax}
\xymatrix{TsTr\ar[r]^{\kappa_{s,r}}\ar[d]_{(T\psi)(T\fii)}&T(sr)\ar[d]^{T(\psi\fii)}&
TtT(sr)\ar[r]^{\kappa_{t,sr}}&T(tsr)\\
Ts'Tr'\ar[r]^-{\kappa_{s',r'}}&T(s'r')
&TtTsTr\ar[r]^-{\kappa_{t,s}-}\ar[u]^{-\kappa_{s,r}}&T(ts)Tr
\ar[u]_{\kappa_{ts,r}}}
\end{equation}
(also: $\kappa_{r,1_X}=1_{Tr}=\kappa_{1_Y,r}$; all unity and
associativity isomorphisms are suppressed).\vspace*{1mm}\\
Furthermore, \emph{we assume that $T(f^\circ)=(Tf)^\circ$ for every map $f$.}

\noindent It follows that whenever $f$ is a set map $\kappa_{s, f}$ is invertible. Its inverse is the composite
\[T(sf) \xrightarrow{-\lambda_{Tf}} T(sf)Tf^\circ Tf \xrightarrow{\kappa_{sf, f^\circ}-} T(sff^\circ) Tf \xrightarrow{T(s\rho_f)-} Ts Tf.\]
Also, $\kappa_{f^\circ, t}$ is invertible, for $t:Z\rel Y$. Its inverse is the composite
\[T(f^\circ t) \xrightarrow{\lambda_{Tf}-} Tf^\circ Tf T(f^\circ t) \xrightarrow{-\kappa_{f, f^\circ t}} Tf^\circ T(f f^\circ t) \xrightarrow{-T(\rho_ft)} Tf^\circ Tt.\]

\item The natural transformations $e:1\to T$, $m:T^2\to T$ of $\Set$
are $\op$-lax in $\VRel$, so that for every $r:X\rel Y$ one
has natural and coherent $2$-cells
\[\alpha=\alpha_r:e_Yr\to
Tre_X,\;\;\beta=\beta_r:m_YT^2r\to Trm_X,\mbox{ as in}\]
\begin{equation}\tag{oplax}\label{eq:oplax}
\xymatrix{  X \ar[r]|-{\object@{|}}^r
\ar[d]_{e_X} & Y \ar[d]^{e_Y}
  \ar@{}[dl]|{\mbox{\large $\stackrel{\alpha}{\Leftarrow}$}} &
  T^2X\ar[r]|-{\object@{|}}^{T^2r}\ar[d]_{m_X}&T^2Y
  \ar@{}[dl]|{\mbox{\large $\stackrel{\beta}{\Leftarrow}$}}\ar[d]^{m_Y} \\
  TX \ar[r]|-{\object@{|}}_{Tr} & TY&
  TX\ar[r]|-{\object@{|}}_{Tr}&TY}
\end{equation}
such that $\alpha_{f}=1_{e_Yf}$, $\beta_{f}=1_{m_YT^2f}$
whenever $r=f$ is a $\Set$-map.
\item The following diagrams
commute (where again we disregard associativity isomorphisms):
\begin{equation}\tag{mon}\label{eq:mon}
\xymatrix@!R=3mm{&&m_YTe_YTr\ar[r]^{-\kappa_{e_Y,r}}\ar[dd]_1&
m_YT(e_Yr)\ar[r]^{-T\alpha_r}&m_YT(Tre_X)\ar[d]^{-\kappa^{-1}_{Tr,e_X}}\\
m_Ye_{TY}Tr\ar[r]^{-\alpha_{Tr}}\ar[d]_1&m_YT^2re_{TX}\ar[d]^{\beta_r-}&
&&m_YT^2rTe_X\ar[d]^{\beta_r-}\\
Tr\ar[r]^-{1}&Trm_Xe_{TX}&Tr\ar[rr]^-{1}&&Trm_XTe_X\\
m_YTm_YT^3r\ar[d]_1\ar[r]^{-\kappa_{m_Y,T^2r}}&
m_YT(m_YT^2r)\ar[r]^{-T\beta_r}&
m_YT(Trm_X)\ar[d]^{-\kappa^{-1}_{Tr,m_X}}\\
m_Ym_{TY}T^3r\ar[d]_{-\beta_{Tr}}&&m_YT^2rTm_X\ar[d]^{\beta_r-}\\
m_YT^2rm_{TX}\ar[r]_{\beta_r-}&Trm_Xm_{TX}\ar[r]_1&
Trm_XTm_X.}\end{equation}
\item One also needs the coherence conditions
\begin{equation}\tag{coh}\label{eq:coh}
\xymatrix@!R=3mm{e_Zsr\ar[r]^{\alpha_s-}\ar[d]_{1}&Tse_Yr\ar[r]^{-\alpha_r}&
TsTre_X\ar[d]^{\kappa_{s,r}-}\\
e_Zsr\ar[rr]^{\alpha_{sr}}&&T(sr)e_X\\
m_ZT^2sT^2r\ar[r]^{\beta_s-}\ar[d]_{-\kappa_{Ts,Tr}}
&Tsm_YT^2r\ar[r]^{-\beta_r}&TsTrm_X\ar[dd]^{\kappa_{s,r}-}\\
m_ZT(TsTr)\ar[d]_{-T\kappa_{s,r}}&&\\
m_ZT^2(sr)\ar[rr]^{\beta_{sr}}&&T(sr)m_X.}
\end{equation}
\item And the following naturality conditions, for all $\fii:r\to r'$,
\begin{equation}\tag{nat}\label{eq:nat}
T\fii e_X\cdot\alpha_r=\alpha_{r'}\cdot e_Y\fii\;\mbox{ and }
T\fii m_X\cdot\beta_r=\beta_{r'}\cdot m_YT^2\fii.
\end{equation}
\end{enumerate}
The op-lax natural transformations $e$ and $m$ induce two lax natural transformations
\[(e^\circ,\hat{\alpha}):T\to\Id_{\VRel}\mbox{ and }(m^\circ,\hat{\beta}):T\to T^2\] on $\VRel$: for each $r:X\relto Y$ we have
\[\xymatrix{TX \ar[r]|-{\object@{|}}^{Tr}
\ar[d]|-{\object@{|}}_{e_X^\circ}\ar@{}[dr]|{\mbox{\large $\stackrel{\hat{\alpha}}{\Rightarrow}$}} & TY \ar[d]|-{\object@{|}}^{e_Y^\circ}
   &
  TX\ar[r]|-{\object@{|}}^{Tr}\ar[d]|-{\object@{|}}_{m_X^\circ}\ar@{}[dr]|{\mbox{\large $\stackrel{\hat{\beta}}{\Rightarrow}$}}&TY
  \ar[d]|-{\object@{|}}^{m_Y^\circ} \\
X \ar[r]|-{\object@{|}}_{r} & Y&
  T^2X\ar[r]|-{\object@{|}}_{T^2r}&T^2Y}\]
where $\hat{\alpha}_r:re_X^\circ\to e_Y^\circ Tr$ and $\hat{\beta}_r:T^2rm_X^\circ\to m_Y^\circ Tr$, are mates of $\alpha_r$ and $\beta_r$ respectively, i.e. they are defined by the composites:
\[\xymatrix{re_X^\circ\ar[r]^-{\lambda_{e_Y}-}&e_Y^\circ e_Y re_X^\circ\ar[r]^-{-\alpha_r-}&e_Y^\circ Tr e_Xe_X^\circ\ar[r]^-{-\rho_{e_X}}&e_Y^\circ Tr\\
T^2rm_X^\circ\ar[r]^-{\lambda_{m_Y}-}&m_Y^\circ m_Y T^2r m_X^\circ\ar[r]^-{-\beta_r-}&m_Y^\circ Tr m_X m_X^\circ\ar[r]^-{-\rho_{m_X}}&m_Y^\circ Tr.}\]

\section{$(\mT,\V)$-categories}\label{sect:three}
Now we define the 2-\cat $\TVCat$ of $(\mT,\V)$-categories, $(\mT,\V)$-functors and transformations between these:
\begin{enumerate}[--]
\item \emph{$(\mT,\V)$-categories} are defined as $(X,a,\eta_a,\mu_a)$, with $X$ a set, $a:TX\rel X$ a $\V$-relation, and $\eta_a$ and $\mu_a$ 2-cells as in the following diagrams:
\[
\xymatrix{X\ar[dr]_{1_X}\ar[r]^{e_X}&TX\ar[d]|-{\object@{|}}^a&TX\ar[d]|-{\object@{|}}_a&T^2X\ar[l]|-{\object@{|}}_{Ta}\ar[d]^{m_X}\\
\ar@{}[ur]|(.75){\mbox{\large $\stackrel{\eta_a}{\Rightarrow}$}}&X&X\ar@{}[ur]|{\mbox{\large
$\stackrel{\mu_a}{\Rightarrow}$}}&
TX;\ar[l]|-{\object@{|}}^a}
\]
furthermore, $\eta_a,\,\mu_a$ provide a
generalized monad structure on $a$, i.e., the following diagrams
must commute (modulo associativity isomorphisms):
\begin{equation}\tag{cat}\label{eq:cat}
\xymatrix{ae_Xa\ar[r]^-{-\alpha_a} & aTae_{TX}\ar[d]^-{\mu_a-}
&aT(ae_X)\ar[r]^-{-\kappa^{-1}_{a,e_X}}&aTaTe_X\ar[d]^{\mu_a-}\\
    a\ar[u]^{\eta_a -}\ar[r]^-{1} & am_Xe_{TX} & a\ar[u]^{-T\eta_a}\ar[r]^-{1}&a m_XTe_X\\
aTaT^2a\ar[r]^{-\kappa_{a,Ta}}\ar[d]_{\mu_a-} &
aT(aTa)\ar[r]^{-T\mu_a} &aT(am_X)\ar[d]^{-\kappa^{-1}_{a,m_X}}\\
am_XT^2a\ar[d]_{-\beta_a}& &aTaTm_X\ar[d]^{\mu_a-}\\
aTam_{TX}\ar[r]^{\mu_a -}&am_Xm_{TX}\ar[r]^{1}&am_XTm_X.}
\end{equation}
We will sometimes denote a $(\mT,\V)$-category $(X,a,\eta_a,\mu_a)$ simply by $(X,a)$.

\item A \emph{$(\mT,\V)$-functor}
$(f,\fii_f):(X,a,\eta_a,\mu_a)\to(Y,b,\eta_b,\mu_b)$ between $(\mT,\V)$-categories is given by a $\Set$-map $f:X\to Y$ equipped with a 2-cell $\varphi_f:fa\to bTf$
\[\label{eq:old6}
\xymatrix{TX\ar[r]^{Tf}\ar[d]|-{\object@{|}}_a&
TY\ar[d]|-{\object@{|}}^b\\
X\ar[r]_f\ar@{}[ur]|{\mbox{\large
$\stackrel{\varphi_f}{\Rightarrow}$}}&Y}
\]
making the following diagrams commute:
\begin{equation}\tag{fun}\label{eq:fun}
\xymatrix@!R=.2mm{f\ar[r]^-{-\eta_a}\ar[dd]_{\eta_b -} & fae_X\ar[dd]^{\fii_f -}\\
\\
    be_Yf\ar[r]^-{1} & bTfe_X\\
    faTa\ar[rr]^{-\mu_a}\ar[dd]_{\varphi_f-}&&fam_X\ar[ddd]^{\fii_f -}\\
&&\\ bTfTa\ar[dd]_{-\kappa_{f,a}}&&\\
&&bTfm_X\ar[ddd]^{1}\\ bT(fa)\ar[dd]_{-T\fii_f}&&\\
&&\\
bT(bTf)\ar[r]^-{-\kappa^{-1}_{b,Tf}}&bTbT^2f\ar[r]^-{\mu_b-}&bm_YT^2f.}
\end{equation}

\item A \emph{$(\mT,\V)$-natural transformation} (or simply a \emph{natural transformation}) between $(\mT,\V)$-functors $(f, \fii_f) \to (g, \fii_g)$ is defined as a 2-cell $\zeta: ga \to bTf$
\[
\xymatrix{TX\ar[r]^{Tf}\ar[d]|-{\object@{|}}_a&
TY\ar[d]|-{\object@{|}}^b\\
X\ar[r]_g\ar@{}[ur]|{\mbox{\large
$\stackrel{\zeta}{\Rightarrow}$}}&Y}
\]
such that the two sides of the following diagram commute
\[\xymatrix@C=.6em@R=1.2em{& \ar[ld]_{\zeta -}gae_Xa&& \ar[ll]_-{-\eta_a-} ga\ar[dddd]_{\zeta} \ar[rr]^-{1}&&ga  \ar[rd]^{\fii_g}&\\
 \ar[d]^{1}bTfe_Xa&&&&&&bTg \ar[d]_{-T(g\eta_a)}\\
\ar[d]^{-\fii_f}be_Yfa&&&&&&bT(gae_X)\ar[d]_{-T(\zeta e_X)}\\
\ar[rd]_-{-\alpha_b-}be_YbTf&&&&&&bT(bTfe_X)\ar[ld]^-{-\kappa^{-1}_{b,e_Yf}}\\
&\ar[rr]_-{\mu_b-}bTbe_{TY}Tf&&bTf&&bTbTe_YTf\ar[ll]^-{\mu_b-}&}
\]
Such a 2-cell $\zeta$ is determined by the 2-cell
\begin{equation}\tag{$\zeta_0$}\label{eq:zeta0}
\xymatrix{(g\ar[r]^-{\zeta_0}& be_Yf)=(g\ar[r]^-{-\eta_a}&gae_X\ar[r]^-{\zeta-}&bTfe_X= be_Yf),}
\end{equation}
from which it can be reconstructed by either side of the above diagram.
\end{enumerate}

The composite of $(\mT, \V)$-functors $(f, \fii_f)$ and $(g, \fii_g)$ is defined by the picture
\[\xymatrix{TX\ar[r]^{Tf}\ar[d]|-{\object@{|}}_a&
TY\ar[d]|-{\object@{|}}^b\ar[r]^{Tg}&TZ\ar[d]|-{\object@{|}}^c\\
X\ar[r]_f\ar@{}[ur]|{\mbox{\large
$\stackrel{\varphi_f}{\Rightarrow}$}}&Y\ar[r]_g\ar@{}[ur]|{\mbox{\large
$\stackrel{\varphi_g}{\Rightarrow}$}}&Z,}\]
that is as $(gf,\varphi_{gf})$, with $\varphi_{gf}=(\varphi_g Tf)(g\varphi_f)$.
The identity $(\mT,\V)$-functor on $(X,a)$ is $(1_X,1_a)$.
The horizontal composition of $(\mT,\V)$-natural transformations $\zeta : (f, \fii_f) \to (g, \fii_g)$ and $\zeta' :  (f', \fii_{f'}) \to (g', \fii_{g'})$ is defined by a picture obtained from the above one by replacing $\fii_f$ and $\fii_g$ with $\zeta$ and $\zeta'$.
The vertical composition of $(\mT,\V)$-natural transformations $\zeta:(f,\fii_f)\to (g,\fii_g)$ and $\zeta':(g,\fii_g)\to(h,\fii_h)$ is defined by the diagram
\[\xymatrix{TX\ar@/_3pc/[dd]_{1_{TX}}^{\mbox{\large ${\stackrel{T\eta_a}{\Rightarrow}}$}}
\ar[d]_{Te_X}\ar[rr]^{Tf}&&TY\ar[d]^{Te_Y}\ar@/^3pc/[ddd]|-{\object@{|}}^{bm_YTe_Y=b}_{\mbox{\large ${\stackrel{\mu_b-}{\Rightarrow}}$}}
\\
T^2X\ar@{}[rrd]|{\mbox{\large ${\stackrel{T\zeta}{\Rightarrow}}$}}\ar[d]|-{\object@{|}}_{Ta}\ar[rr]^{T^2f}&&T^2Y\ar[d]|-{\object@{|}}^{Tb}\\
TX\ar@{}[rrd]|{\mbox{\large ${\stackrel{\zeta'}{\Rightarrow}}$}}\ar[rr]^{Tg}\ar[d]|-{\object@{|}}_a&&TY\ar[d]|-{\object@{|}}^b\\
X\ar[rr]^h&&Y.}\]
The identity natural transformation on a $(\mT, \V)$-functor $(f, \fii_f)$ is the 2-cell $\fii_f$ itself.

The definitions of horizontal and vertical compositions can be naturally stated in terms of the alternative definition (\ref{eq:zeta0}) of $(\mT,\V)$-natural transformation too.

\bigskip
When $\mT$ is the identity monad, identically extended to $\VRel$, the category $\TVCat$ is exactly the 2-category $\VCat$ of $\V$-categories, $\V$-functors and \V-natural transformations.

Next we summarize briefly our two main examples. In the first example,  $\V=\two$ and $\mT$ is the ultrafilter monad together with a suitable extension to $\two\mbox{-}\Rel = \Rel$. In this case $\TVCat$ is the category of topological spaces and continuous maps. In the second example, $\V=\Set$ and $\mT$ is the free-monoid monad with a suitable extension to $\Set\mbox{-}\Rel = \rm\bf{Span}$. In this case $\TVCat$ is the category of multicategories and multifunctors. For details on these examples, as well as for other examples, see \cite{CT03, HST14}.

\bigskip
For any $\mT$ there is an adjunction of 2-functors:
\begin{equation}\tag{adj}\label{eq:adj}
\TVCat\adjunct{A^\circ}{A_e}\VCat.
\end{equation}
$A_e$ is the algebraic functor associated with $e$, that is, for any $(\mT,\V)$-category $(X,a,\eta_a,\mu_a)$, $(\mT,\V)$-functor $(f,\varphi_f)$ and $(\mT,\V)$-natural transformation $\zeta:(f,\varphi_f)\to(g,\varphi_g)$, $A_e(X,a,\eta_a,\mu_a)=(X,ae_X,\eta_a,\overline{\mu}_a)$, where
\[\xymatrix{(ae_Xae_X\ar[r]^-{\overline{\mu}_a} &ae_X)=(ae_Xae_X\ar[r]^-{-\alpha_a-}&aTae_{TX}e_X\ar[r]^-{\mu_a-}&am_Xe_{TX}e_X=ae_X),}\]
$A_e(f,\varphi_f)=(f,\varphi_fe_X)$ and $A_e(\zeta)=\zeta e_X$ (see \cite{CT03} for details).

$A^\circ$ is defined as follows. For a $\V$-category $(Z,c,\eta_c,\mu_c)$, $A^\circ(Z,c,\eta_c,\mu_c)$ is the $(\mT,\V)$-category $(Z,c^\sharp,\eta_{c^\sharp},\mu_{c^\sharp})$ where $c^\sharp=e_Z^\circ Tc$, while $\eta_{c^\sharp}:1\to e_Z^\circ Tce_Z$ and $\mu_{c^\sharp}:e_Z^\circ Tc T(e_Z^\circ Tc)\to e_Z^\circ Tc m_Z$ are defined by the composites
\[\xymatrix{1\ar[r]^-{\lambda_{e_Z}}&e_Z^\circ e_Z\ar[r]^-{-T\eta_c-}&e_Z^\circ Tc e_Z}\]
\[\xymatrix@R=1.5em{
T^2Z \ar[dd]_{m_Z} \ar[r]|-{\object@{|}}^{T^2c} \ar@{}[rdd]|{\mbox{\large $\stackrel{\beta_c}{\Leftarrow}$}} \ar@/^2.5pc/[rrr]|-{\object@{|}}^{T(e_Z^\circ Tc)}_{\mbox{\large $\stackrel{\kappa^{-1}_{e_Z^\circ, Tc}}{\Leftarrow}$}} &T^2Z  \ar[rr]|-{\object@{|}}^{Te^\circ_Z} \ar[d]_{1_{T^2Z}} \ar@{}[rd]|{\mbox{\large $\stackrel{\rho_{Te_Z}}{\Leftarrow}$}} &&TZ  \ar[d]|-{\object@{|}}_{Tc} \ar@/^1pc/[lld]^{Te_Z}\\
&T^2Z \ar[d]_{m_Z} && TZ \ar[d]|-{\object@{|}}_{e^\circ_Z}\\
TZ \ar[r]|-{\object@{|}}_{Tc} \ar@/_2pc/[rr]|-{\object@{|}}_{Tc}^{\mbox{\large $\stackrel{T\mu_c \kappa_{c,c}}{\Leftarrow}$}}& TZ  \ar[r]|-{\object@{|}}_{Tc} &TZ \ar[r]|-{\object@{|}}_{e^\circ_Z}&Z.
}
\]
 For a  $\V$-functor $(f, \fii_f) : (Z, c) \to (Z', c')$, $A^\circ(f, \fii_f)$ is defined by the diagram
\[\xymatrix{TZ\ar[r]|-{\object@{|}}^{Tc}\ar[d]_{Tf}&
TZ\ar[r]|-{\object@{|}}^{e^\circ_Z}\ar[d]^{Tf}&Z\ar[d]^f\\
TZ'\ar[r]|-{\object@{|}}_{Tc'}\ar@{}[ur]|{\mbox{\large
$\stackrel{T\varphi_f}{\Leftarrow}$}}&TZ'\ar[r]|-{\object@{|}}_{e^\circ_{Z'}}\ar@{}[ur]|{\mbox{\large
$\stackrel{}{\Leftarrow}$}}&Z',}\]
wherein the right 2-cell is the mate of the identity 2-cell $1_{Tfe_Z=e_{Z'}f}$. On $\V$-natural transformations $A^\circ$ is defined by a similar diagram. By direct verifications $A^\circ$ is indeed a 2-functor, and as already stated we have:

\begin{prop}
$A^\circ$ is a left 2-adjoint to $A_e$.
\end{prop}

\begin{proof}
The unit of the adjunction has the component at a $\V$-category $(Z,c)$ given by a $\V$-functor consisting of $1_Z$ and the 2-cell
\[\xymatrix{c \ar[r]^-{\lambda_{e_Z}-}& e_Z^\circ e_Zc \ar[r]^-{- \alpha_c} &e_Z^\circ Tc e_Z.}\]
The counit of the adjunction has the component at a $(\mT, \V)$-category $(X, a)$ given by a $(\mT, \V)$-functor consisting of $1_X$ and the 2-cell
\[\xymatrix{e_X^\circ T(ae_X) \ar[r]^-{-\kappa^{-1}_{a, e_X}}& e_X^\circ Ta Te_X \ar[r]^-{\eta_a -}& a e_X e_X^\circ TaTe_X \ar[r]^-{-\rho_{e_X}-}  &aTaTe_X\ar[r]^-{\mu_a-} &am_XTe_X = a.}\]
The triangle identities are then directly verified.
\end{proof}

The next proposition is a $(\mT, \V)$-categorical analogue of the ordinary- and enriched-categorical fact that an adjunction between functors induces isomorphisms between hom-sets/-objects.

\begin{prop}\label{th:adj}
Given an adjunction  $(f, \fii_f) \dashv (g, \fii_g) : (X, a) \rightarrow (Y, b)$ in the 2-category $\TVCat$, there is an isomorphism:
\[g^\circ a \cong bTf.\]
\end{prop}

\begin{proof}
The unit and the  counit of the given adjunction are $(\mT, \V)$-natural transformations $(1_X, 1_a) \to (g, \fii_g)(f, \fii_f)$ and $(f, \fii_f)(g, \fii_g) \to (1_Y, 1_b)$. These are given by 2-cells $\upsilon_0 : gf \to ae_X$ and $\epsilon_0: 1_Y \to be_Yfg$ respectively. Define a 2-cell $bTf \to  g^\circ a$ by
\[
\xymatrix@=\sizexy{
TX \ar[rr]^{Tf} \ar[d]_{Te_X} \ar `l^d[dd]`^r[dd]_{1_{TX}} [dd] \ar@{}[rrd]|{\mbox{\large $\stackrel{-T\upsilon_0-}{\Leftarrow}$}}&& TY \ar[rr]|-{\object@{|}}^{b} \ar[d]_{Tg} \ar@{}[rrd]|{\mbox{\large$\stackrel{\fii_g}{\Leftarrow}$ }}&& Y \ar[d]_{g} \ar[rr]^{1_Y}  \ar@{}[rd]+UUU|{\mbox{\large $\stackrel{\lambda_g}{\Leftarrow}$}}&& Y,\\
T^2X \ar[d]_{m_X} \ar[rr]|-{\object@{|}}^{Ta} \ar@{}[rrrd]|{\mbox{\large $\stackrel{\mu_a}{\Leftarrow}$}}&& TX \ar[rr]|-{\object@{|}}^{a}&& X \ar@<-.3em>[rru]_{g^\circ}&& \\
TX \ar@/_1.2pc/[urrrr]|-{\object@{|}}_{a}  && && && \\
}
\]
wherein the blank symbols stand for the obvious instances of $\kappa$ or $\kappa^{-1}$.
 In the opposite direction define a 2-cell $g^\circ a \to bTf$ by
\[
\xymatrix@=\sizexy{
&& && && Y \ar[d]^{g} \ar@{}[ld]|{\mbox{\large $\stackrel{\rho_g}{\Leftarrow}$}} \ar@/^2pc/[dddd]^{1_Y}_{\mbox{\large $\stackrel{\epsilon_0}{\Leftarrow}$}}\\
TX \ar[rrrr]|-{\object@{|}}_{a} \ar[d]_{Tf} \ar@{}[rrrrd]|{\mbox{\large $\stackrel{\fii_f}{\Leftarrow}$}} && && X \ar[d]^{f} \ar[rr]^{1_X} \ar@<.3em>[rru]^{g^\circ} && X \ar[d]^{f} \\
TY\ar[rrrr]|-{\object@{|}}^{b} \ar[d]_{e_{TY}} \ar `l^d[dd]`^r[dd]_{1_{TY}} [dd] \ar@{}[rrrrd]|{\mbox{\large $\stackrel{\alpha_b}{\Leftarrow}$}}&&  && Y \ar[d]_{e_Y} && Y \ar[d]^{e_Y} \\
T^2Y \ar[rrrr]|-{\object@{|}}_{Tb}\ar[d]_{m_Y} \ar@{}[rrrrd]|{\mbox{\large $\stackrel{\mu_b}{\Leftarrow}$}}&&  && TY \ar[d]|-{\object@{|}}_{b} && TY \ar[d]|-{\object@{|}}^{b}\\
TY \ar[rrrr]|-{\object@{|}}_{b} && && Y \ar[rr]_{1_Y} && Y.
}
\]
These two 2-cells are inverses to each other. The following calculation shows that the equality $(bTf \to  g^\circ a \to  bTf) = 1_{bTf}$ holds. The remaining equation is proved using analogous arguments. Pasting the first diagram on top of the second, and using the equation $(\rho_gg) (g\lambda_g)= 1_g$ we obtain
\[
\xymatrix@=\sizexy{
TX \ar[rr]^{Tf} \ar[d]_{Te_X} \ar `l^d[dd]`^r[dd]_{1_{TX}} [dd] \ar@{}[rrd]|{\mbox{\large $\stackrel{-T\upsilon_0-}{\Leftarrow}$}}&& TY \ar[rr]|-{\object@{|}}^{b} \ar[d]_{Tg} \ar@{}[rrd]|{\mbox{\large$\stackrel{\fii_g}{\Leftarrow}$ }}&& Y  \ar[d]_{g} \ar@/^2pc/[ddddd]^{1_Y}_{\mbox{\large $\stackrel{\epsilon_0}{\Leftarrow}$}} \\
T^2X \ar[d]_{m_X} \ar[rr]|-{\object@{|}}^{Ta} \ar@{}[rrrd]|{\mbox{\large $\stackrel{\mu_a}{\Leftarrow}$}}&& TX \ar[rr]|-{\object@{|}}^{a}&& X \ar[dd]_{f} \\
TX \ar@/_1.2pc/[urrrr]|-{\object@{|}}_{a} \ar[d]_{Tf} \ar@{}[rrrrd]|{\mbox{\large $\stackrel{\fii_f}{\Leftarrow}$}} && &&  \\
TY\ar[rrrr]|-{\object@{|}}^{b} \ar[d]_{e_{TY}} \ar `l^d[dd]`^r[dd]_{1_{TY}} [dd] \ar@{}[rrrrd]|{\mbox{\large $\stackrel{\alpha_b}{\Leftarrow}$}}&&  && Y \ar[d]_{e_Y} \\
T^2Y \ar[rrrr]|-{\object@{|}}_{Tb}\ar[d]_{m_Y} \ar@{}[rrrrd]|{\mbox{\large $\stackrel{\mu_b}{\Leftarrow}$}}&&  && TY \ar[d]|-{\object@{|}}_{b} \\
TY \ar[rrrr]|-{\object@{|}}_{b} && && Y;
}
\]
using (\ref{eq:fun}) for $(f, \fii_f)$ we get
\[
\xymatrix@=\sizexy{
&TX \ar[rr]^{Tf} \ar[d]_{Te_X} \ar@/_1.5pc/[ldd]_{1_{TX}} \ar@{}[rrd]|{\mbox{\large $\stackrel{-T\upsilon_0-}{\Leftarrow}$}}&& TY \ar[rr]|-{\object@{|}}^{b} \ar[d]_{Tg} \ar@{}[rrd]|{\mbox{\large$\stackrel{\fii_g}{\Leftarrow}$ }}&& Y  \ar[d]_{g} \ar@/^2pc/[ddddd]^{1_Y}_{\mbox{\large $\stackrel{\epsilon_0}{\Leftarrow}$}} \\
&T^2X \ar[d]_{T^2f} \ar[rr]|-{\object@{|}}^{Ta} \ar@{}[rrd]|{\mbox{\large $\stackrel{-T\fii_f-}{\Leftarrow}$}} \ar[ld]_{m_X}&& TX \ar[rr]|-{\object@{|}}^{a} \ar[d]_{Tf} \ar@{}[rrd]|{\mbox{\large $\stackrel{\fii_f}{\Leftarrow}$}}&& X \ar[d]_{f} \\
TX \ar[rd]_{Tf} &T^2Y \ar[rr]|-{\object@{|}}^{Tb} \ar[d]_{m_Y} \ar@{}[rrrd]|{\mbox{\large $\stackrel{\mu_b}{\Leftarrow}$}} &&TY \ar[rr]|-{\object@{|}}^{b} && Y \ar[dd]_{e_Y}\\
&TY \ar[d]_{e_{TY}} \ar@/_1.2pc/[urrrr]|-{\object@{|}}_{b} \ar `l^d[dd]`^r[dd]_{1_{TY}} [dd] \ar@{}[rrrrd]|{\mbox{\large $\stackrel{\alpha_b}{\Leftarrow}$}}&&  &&  \\
&T^2Y \ar[rrrr]|-{\object@{|}}_{Tb}\ar[d]_{m_Y} \ar@{}[rrrrd]|{\mbox{\large $\stackrel{\mu_b}{\Leftarrow}$}}&&  && TY \ar[d]|-{\object@{|}}_{b} \\
&TY \ar[rrrr]|-{\object@{|}}_{b} && && Y.
}
\]
Then, using naturality of $\alpha$ we obtain
\[
\xymatrix@=\sizexy{
&TX \ar[rr]^{Tf} \ar[d]_{Te_X} \ar@/_1.5pc/[ldd]_{1_{TX}} \ar@{}[rrd]|{\mbox{\large $\stackrel{-T\upsilon_0-}{\Leftarrow}$}}&& TY \ar[rr]|-{\object@{|}}^{b} \ar[d]_{Tg} \ar@{}[rrd]|{\mbox{\large$\stackrel{\fii_g}{\Leftarrow}$ }}&& Y  \ar[d]_{g} \ar@/^2pc/[ddddd]^{1_Y}_{\mbox{\large $\stackrel{\epsilon_0}{\Leftarrow}$}} \\
&T^2X \ar[d]_{T^2f} \ar[rr]|-{\object@{|}}^{Ta} \ar@{}[rrd]|{\mbox{\large $\stackrel{-T\fii_f-}{\Leftarrow}$}} \ar[ld]_{m_X}&& TX \ar[rr]|-{\object@{|}}^{a} \ar[d]_{Tf} \ar@{}[rrd]|{\mbox{\large $\stackrel{\fii_f}{\Leftarrow}$}}&& X \ar[d]_{f} \\
TX \ar[d]_{Tf} &T^2Y \ar[ld]_{m_Y} \ar[rr]|-{\object@{|}}^{Tb} \ar[d]_{e_{T^2Y}} \ar@{}[rrd]|{\mbox{\large $\stackrel{\alpha_{Tb}}{\Leftarrow}$}} &&TY \ar[d]_{e_{TY}} \ar[rr]|-{\object@{|}}^{b} \ar@{}[rrd]|{\mbox{\large $\stackrel{\alpha_{b}}{\Leftarrow}$}}&& Y \ar[d]_{e_Y}\\
TY \ar@/_1.5pc/[rdd]_{1_{TY}} \ar[rd]_{e_{TY}}&T^3Y \ar[rr]|-{\object@{|}}^{T^2b} \ar[d]_{Tm_Y} \ar@{}[rrrd]|{\mbox{\large $\stackrel{-T\mu_b-}{\Leftarrow}$}}&& T^2Y \ar[rr]|-{\object@{|}}^{Tb}&& TY  \ar[dd]|-{\object@{|}}_{b}\\
&T^2Y \ar[d]_{m_Y}  \ar@/_1.2pc/[urrrr]|-{\object@{|}}_{Tb}\ar@{}[rrrrd]|{\mbox{\large $\stackrel{\mu_b}{\Leftarrow}$}}&&  &&  \\
&TY \ar[rrrr]|-{\object@{|}}_{b} && && Y,
}
\]
and using the associativity axiom in (\ref{eq:cat}) for $\mu_b$  we get
\[
\xymatrix@=\sizexy{
&TX \ar[rr]^{Tf} \ar[d]_{Te_X} \ar `l /25pt[ld]`[ddddd]_{Tf} [ddddd]  \ar@{}[rrd]|{\mbox{\large $\stackrel{-T\upsilon_0-}{\Leftarrow}$}}&& TY \ar[rr]|-{\object@{|}}^{b} \ar[d]_{Tg} \ar@{}[rrd]|{\mbox{\large$\stackrel{\fii_g}{\Leftarrow}$ }}&& Y  \ar[d]_{g} \ar@/^2pc/[dddd]^{1_Y}_{\mbox{\large $\stackrel{\epsilon_0}{\Leftarrow}$}} \\
&T^2X \ar[d]_{T^2f} \ar[rr]|-{\object@{|}}^{Ta} \ar@{}[rrd]|{\mbox{\large $\stackrel{-T\fii_f-}{\Leftarrow}$}} && TX \ar[rr]|-{\object@{|}}^{a} \ar[d]_{Tf} \ar@{}[rrd]|{\mbox{\large $\stackrel{\fii_f}{\Leftarrow}$}}&& X \ar[d]_{f} \\
&T^2Y \ar[rr]|-{\object@{|}}^{Tb} \ar[d]_{e_{T^2Y}} \ar@{}[rrd]|{\mbox{\large $\stackrel{\alpha_{Tb}}{\Leftarrow}$}} &&TY \ar[d]_{e_{TY}} \ar[rr]|-{\object@{|}}^{b} \ar@{}[rrd]|{\mbox{\large $\stackrel{\alpha_{b}}{\Leftarrow}$}}&& Y \ar[d]_{e_Y}\\
&T^3Y \ar[rr]|-{\object@{|}}^{T^2b} \ar[d]+<-1em,.8em>_{Tm_Y} \ar@<.5em>[d]^{m_{TY}} \ar@{}[rrd]|{\mbox{\large $\stackrel{-T\mu_b-}{\Leftarrow}$}}&& T^2Y \ar[d]_{m_Y} \ar[rr]|-{\object@{|}}^{Tb} \ar@{}[rrd]|{\mbox{\large $\stackrel{\mu_b}{\Leftarrow}$}} && TY  \ar[d]|-{\object@{|}}_{b}\\
&T^2Y \enspace T^2Y \ar@<.5em>[d]^{m_Y} \ar[rr]|-{\object@{|}}^{Tb} \ar@{}[rrrd]|{\mbox{\large $\stackrel{\mu_b}{\Leftarrow}$}}&& TY \ar[rr]|-{\object@{|}}^{b} && Y. \\
& TY \ar@/_1.2pc/[urrrr]|-{\object@{|}}_{b} \ar[u]+<-1em,-.8em>;[]_{m_Y}&& &&
}
\]
From (\ref{eq:mon}) we obtain
\[
\xymatrix@=\sizexy{
&TX \ar[rr]^{Tf} \ar[d]_{Te_X} \ar `l /25pt[ld]`[ddddd]_{Tf} [ddddd]  \ar@{}[rrd]|{\mbox{\large $\stackrel{-T\upsilon_0-}{\Leftarrow}$}}&& TY \ar[rr]|-{\object@{|}}^{b} \ar[d]_{Tg} \ar@{}[rrd]|{\mbox{\large$\stackrel{\fii_g}{\Leftarrow}$ }}&& Y  \ar[d]_{g} \ar@/^2pc/[dddd]^{1_Y}_{\mbox{\large $\stackrel{\epsilon_0}{\Leftarrow}$}} \\
&T^2X \ar[d]_{T^2f} \ar[rr]|-{\object@{|}}^{Ta} \ar@{}[rrd]|{\mbox{\large $\stackrel{-T\fii_f-}{\Leftarrow}$}} && TX \ar[rr]|-{\object@{|}}^{a} \ar[d]_{Tf} \ar@{}[rrd]|{\mbox{\large $\stackrel{\fii_f}{\Leftarrow}$}}&& X \ar[d]_{f} \\
&T^2Y \ar@/^1.5pc/[dd]^(0.4){1_{T^2Y}} \ar[rr]|-{\object@{|}}^{Tb} \ar[d]_{e_{T^2Y}} &&TY \ar@/_1.5pc/[dd]_(0.6){1_{T^2Y}} \ar[d]^{e_{TY}} \ar[rr]|-{\object@{|}}^{b} \ar@{}[rrd]|{\mbox{\large $\stackrel{\alpha_{b}}{\Leftarrow}$}}&& Y \ar[d]_{e_Y}\\
&T^3Y \ar[d]_{m_{TY}} && T^2Y \ar[d]^{m_Y} \ar[rr]|-{\object@{|}}^{Tb} \ar@{}[rrd]|{\mbox{\large $\stackrel{\mu_b}{\Leftarrow}$}} && TY  \ar[d]|-{\object@{|}}_{b}\\
&T^2Y \ar[d]_{m_Y} \ar[rr]|-{\object@{|}}^{Tb} \ar@{}[rrrd]|{\mbox{\large $\stackrel{\mu_b}{\Leftarrow}$}}&& TY \ar[rr]|-{\object@{|}}^{b} && Y, \\
& TY \ar@/_1.2pc/[urrrr]|-{\object@{|}}_{b}&& &&
}
\]
and the axiom of a $(\mT, \V)$-natural transformation for $\epsilon_0$ gives
\[
\xymatrix@=\sizexy{
&TX \ar[rr]^{Tf} \ar[d]_{Te_X} \ar `l /25pt[ld]`[ddddd]_{Tf} [ddddd]  \ar@{}[rrd]|{\mbox{\large $\stackrel{-T\upsilon_0-}{\Leftarrow}$}}&& TY \ar[d]^{Tg} \ar@/^1.5pc/[rrddd]^{1_{TY}}&& \\
&T^2X \ar[d]_{T^2f} \ar[rr]|-{\object@{|}}^{Ta} \ar@{}[rrd]|{\mbox{\large $\stackrel{-T\fii_f-}{\Leftarrow}$}} && TX \ar[d]^{Tf} && \\
&T^2Y \ar[dd]_{1_{T^2Y}} \ar[rr]|-{\object@{|}}^{Tb} &&TY \ar@/_1.5pc/[dd]_(0.6){1_{T^2Y}} \ar[d]^{Te_Y} \ar@{}[rr]|{\mbox{\large $\stackrel{-T\epsilon_0-}{\Leftarrow}$}}&& \\
& && T^2Y \ar[d]^{m_Y} \ar[rr]|-{\object@{|}}^{Tb} \ar@{}[rrd]|{\mbox{\large $\stackrel{\mu_b}{\Leftarrow}$}} && TY  \ar[d]|-{\object@{|}}_{b}\\
&T^2Y \ar[d]^{m_Y} \ar[rr]|-{\object@{|}}^{Tb} \ar@{}[rrrd]|{\mbox{\large $\stackrel{\mu_b}{\Leftarrow}$}}&& TY \ar[rr]|-{\object@{|}}^{b} && Y. \\
& TY \ar@/_1.2pc/[urrrr]|-{\object@{|}}_{b}&& &&
}
\]
Using (\ref{eq:mon}) again we obtain
\[
\xymatrix@=\sizexy{
&TX \ar[rr]^{Tf} \ar[d]_{Te_X} \ar `l /25pt[ld]`[ddddd]_{Tf} [ddddd]  \ar@{}[rrd]|{\mbox{\large $\stackrel{-T\upsilon_0-}{\Leftarrow}$}}&& TY \ar[d]^{Tg} \ar@/^1.5pc/[rrddd]^{1_{TY}}&& \\
&T^2X \ar[d]_{T^2f} \ar[rr]|-{\object@{|}}^{Ta} \ar@{}[rrd]|{\mbox{\large $\stackrel{-T\fii_f-}{\Leftarrow}$}} && TX \ar[d]^{Tf} && \\
&T^2Y \ar@/_1.3pc/[dd]_{1_{T^2Y}} \ar[d]^{Te_{TY}} \ar[rr]|-{\object@{|}}^{Tb} \ar@{}[rrd]|{\mbox{\large $\stackrel{-T\alpha_b-}{\Leftarrow}$}} &&TY \ar[d]^{Te_Y} \ar@{}[rr]|{\mbox{\large $\stackrel{-T\epsilon_0-}{\Leftarrow}$}}&& \\
&T^3Y \ar[d]^{m_{TY}} \ar[rr]|-{\object@{|}}^{T^2b} \ar@{}[rrd]|{\mbox{\large $\stackrel{\beta_b}{\Leftarrow}$}}&& T^2Y \ar[d]^{m_Y} \ar[rr]|-{\object@{|}}^{Tb} \ar@{}[rrd]|{\mbox{\large $\stackrel{\mu_b}{\Leftarrow}$}} && TY  \ar[d]|-{\object@{|}}_{b}\\
&T^2Y \ar[d]^{m_Y} \ar[rr]|-{\object@{|}}^{Tb} \ar@{}[rrrd]|{\mbox{\large $\stackrel{\mu_b}{\Leftarrow}$}}&& TY \ar[rr]|-{\object@{|}}^{b} && Y, \\
& TY \ar@/_1.2pc/[urrrr]|-{\object@{|}}_{b}&& &&
}
\]
and using associativity of $\mu_b$ again we get
\[
\xymatrix@=\sizexy{
&TX \ar[rr]^{Tf} \ar[d]_{Te_X} \ar `l /25pt[ld]`[ddddd]_{Tf} [ddddd]  \ar@{}[rrd]|{\mbox{\large $\stackrel{-T\upsilon_0-}{\Leftarrow}$}}&& TY \ar[d]^{Tg} \ar@/^1.5pc/[rrddd]^{1_{TY}}&& \\
&T^2X \ar[d]_{T^2f} \ar[rr]|-{\object@{|}}^{Ta} \ar@{}[rrd]|{\mbox{\large $\stackrel{-T\fii_f-}{\Leftarrow}$}} && TX \ar[d]^{Tf} && \\
&T^2Y \ar[d]_{Te_{TY}} \ar[rr]|-{\object@{|}}^{Tb} \ar@{}[rrd]|{\mbox{\large $\stackrel{-T\alpha_b-}{\Leftarrow}$}} &&TY \ar[d]^{Te_Y} \ar@{}[rr]|{\mbox{\large $\stackrel{-T\epsilon_0-}{\Leftarrow}$}}&& \\
&T^3Y \ar[d]+<-1em,.8em>_{m_{TY}} \ar@<.5em>[d]^{Tm_Y} \ar[rr]|-{\object@{|}}^{T^2b} \ar@{}[rrd]|{\mbox{\large $\stackrel{-T\mu_b-}{\Leftarrow}$}}&& T^2Y \ar[rr]|-{\object@{|}}^{Tb} && TY  \ar[d]|-{\object@{|}}_{b}\\
&T^2Y \enspace T^2Y \ar@<.5em>[d]^{m_Y} \ar@/^-1.2pc/[rrrru]|-{\object@{|}}^{Tb} \ar@{}[rrrd]|{\mbox{\large $\stackrel{\mu_b}{\Leftarrow}$}}&&  && Y. \\
& TY \ar[u]+<-1em,-.8em>;[]_{m_Y} \ar@/_1.2pc/[urrrr]|-{\object@{|}}_{b}&& &&
}
\]
Now, one of the triangle equations satisfied by the unit $\upsilon_0$ and the counit $\epsilon_0$ of our adjunction gives us
\[
\xymatrix@=\sizexy{
&TX \ar[rr]^{Tf} \ar[d]_{Te_X} \ar `l /25pt[ld]`[ddddd]_{Tf} [ddddd] \ar@{}[rrrrddd]|{\mbox{\large $\stackrel{-T\eta_b-}{\Leftarrow}$}} && TY \ar[lldd]^{Te_Y} \ar@/^1.5pc/[rrddd]^{1_{TY}}&& \\
&T^2X \ar[d]_{T^2f} &&  && \\
&T^2Y \ar[d]_{Te_{TY}} \ar[rrrrd]|-{\object@{|}}^{Tb} \ar@/^1.5pc/[dd]^{1_{TY}} && && \\
&T^3Y  \ar[d]_{Tm_Y} && && TY  \ar[d]|-{\object@{|}}_{b}\\
&T^2Y \ar[d]_{m_Y} \ar@/^-1.2pc/[rrrru]|-{\object@{|}}^{Tb} \ar@{}[rrrd]|{\mbox{\large $\stackrel{\mu_b}{\Leftarrow}$}}&&  && Y, \\
& TY \ar@/_1.2pc/[urrrr]|-{\object@{|}}_{b}&& &&
}
\]
and finally, by the unity axiom in (\ref{eq:cat}), this equals to
\[
\xymatrix@=1.5em{
&TX \ar[rr]^{Tf} \ar[dd]_{Tf} && TY \ar[ld]_{Te_Y} \ar[dd]|-{\object@{|}}^{b} \\
&&T^2Y \ar[ld]_{m_Y}  &  \\
& TY \ar[rr]|-{\object@{|}}^{b}&& Y,
}
\]
which is the identity map $1_{bTf}$.

We leave it to the reader to verify the equality $(g^\circ a \to  bTf \to g^\circ a) = 1_{g^\circ a}$.
\end{proof}

\section{$\mT$ as a $\VCat$ monad}\label{sect:VCatMonad}

In this section we show that the properties of the lax extension of the $\Set$-monad $\mT$ to $\VRel$ allow us to extend $\mT$ to $\VCat$. Straightforward calculations show that:
\begin{lemma}
\begin{enumerate}[\em (1)]
\item If $(X,a,\eta_a,\mu_a)$ is a $\V$-category, then $(TX,Ta,T\eta_a,T\mu_a\kappa_{a,a})$ is a $\V$-category.
\item If $(f,\fii_f):(X,a,\eta_a,\mu_a)\to(Y,b,\eta_b,\mu_b)$ is a $\V$-functor, then $(Tf,\fii_{Tf}):(TX,Ta)\to(TY,Tb)$, where $\fii_{Tf}:=\kappa^{-1}_{b,f}  \, T\fii_f \,  \kappa_{f,a}$, is a $\V$-functor as well.
\item If $\zeta : (f, \fii_f) \to (g, \fii_g)$ is a $\V$-natural transformation, then so is $\kappa^{-1}_{b,f} T\zeta  \, \kappa_{g,a}:(Tf, \fii_{Tf}) \to (g, \fii_{Tg})$.
\end{enumerate}
\end{lemma}
\noindent These assignments define an endo 2-functor on $\VCat$ that we denote again by $T:\VCat\to\VCat$. The 2-cells $\alpha, \beta$ of the oplax natural transformations $e, m$ on $\VRel$ equip $e$ and $m$ so that they become natural transformations in $\VCat$, as we show next.

\begin{lemma}
For each $\V$-category $(X,a)$:
\begin{enumerate}[\em (1)]
\item $(e_X,\alpha_a):(X,a)\to(TX,Ta)$ is a $\V$-functor;
\item $(m_X,\beta_a):(T^2X,T^2a)\to(TX,Ta)$ is a $\V$-functor.
\end{enumerate}
\end{lemma}
\begin{proof}
To check that the diagrams
\[\xymatrix@!C=15ex{e_X\ar[r]^{-\eta_a}\ar[d]_{T\eta_a-}&e_Xa\ar[ld]^{\alpha_a}&
m_X\ar[r]^-{-\eta_{T^2a}}\ar[d]_{\eta_{Ta}-}&m_XT^2a\ar[ld]^{\beta_a}\\
Tae_X&&Tam_X}\]
commute one uses the naturality conditions (\ref{eq:nat}) with respectively $\varphi=\eta$ and $\varphi=\beta$.
For the diagrams
\[\xymatrix@!C=100pt{e_Xaa\ar[rr]^{-\mu_a}\ar[d]_{\alpha_a-}\ar[rdd]^{\alpha_{a,a}}&&e_Xa\ar[dd]^{\alpha_a}\\
Tae_Xa\ar[d]_{-\alpha_a}\\
TaTae_X\ar@{}[rruu]^(0.2){\framebox{1}}\ar@{}[rruu]|{\framebox{2}}\ar[r]^{\kappa_{a,a}-}&T(aa)e_X\ar[r]^{T\mu_a-}&Tae_X\\
m_XT^2aT^2a\ar[r]^{-\kappa_{Ta,Ta}}\ar[d]_{\beta_a-}&m_XT(TaTa)\ar[r]^{-T\kappa_{a,a}}&m_XT^2(aa)
\ar[r]^{-T^2\mu_a}\ar[dd]^{\beta_{aa}}&m_XT^2a\ar[dd]^{\beta_a}\\
Tam_XT^2a\ar[d]_{-\beta_a}\\
TaTam_X\ar@{}[rruu]|{\framebox{3}}\ar[rr]^{\kappa_{a,a}-}&&T(aa)m_X\ar@{}[ruu]|{\framebox{4}}\ar[r]^{T\mu_a-}&Tam_X,}\]
commutativity of $\framebox{1}$ and $\framebox{3}$ follows from the coherence conditions (\ref{eq:coh}), while commutativity of $\framebox{2}$ and $\framebox{4}$ follows from the naturality conditions (\ref{eq:nat}).
\end{proof}

\begin{lemma}
For each $\V$-category $(X,a)$, let $e_{(X,a)}=(e_X,\alpha_a)$ and $m_{(X,a)}=(m_X,\beta_a)$.
\begin{enumerate}[\em (1)]
\item $e=(e_{(X,a)})_{(X,a)\in\VCat}:\Id_{\VCat}\to T$ is a 2-natural transformation.
\item $m=(m_{(X,a)})_{(X,a)\in\VCat}:T^2\to T$ is a 2-natural transformation.
\end{enumerate}
\end{lemma}
\begin{proof}
To check that, in the diagrams
\[\xymatrix{&X\ar[rr]^{e_X}\ar[ld]|-{\object@{|}}_a\ar[dd]^>>>>>>{f}&&TX\ar[ld]|-{\object@{|}}^{Ta}\ar[dd]^{Tf}&&T^2X\ar[rr]^{m_X}\ar[ld]|-{\object@{|}}_{T^2a}
\ar[dd]^>>>>>>{T^2f}&&TX\ar[ld]|-{\object@{|}}^{Ta}\ar[dd]^{Tf}\\
X\ar@{}[dr]|{\mbox{\large $\Downarrow$}\varphi_{f}}\ar@{}[urrr]|{\mbox{\large $\stackrel{\alpha_a}{\Rightarrow}$}}\ar[rr]^(0.65){e_X}\ar[dd]_f&&TX\ar@{}[dr]|{\mbox{\large $\Downarrow$}\varphi_{Tf}}\ar[dd]^(0.65){Tf}&&
T^2X\ar@{}[dr]|{\mbox{\large $\Downarrow$}\varphi_{T^2f}}\ar@{}[urrr]|{\mbox{\large $\stackrel{\beta_a}{\Rightarrow}$}}\ar[dd]_{T^2f}\ar[rr]^(0.65){m_X}&&TX\ar@{}[dr]|{\mbox{\large $\Downarrow$}\varphi_{Tf}}\ar[dd]^(0.65){Tf}\\
&Y\ar[rr]^(0.35){e_Y}\ar[ld]|-{\object@{|}}_b&&TY\ar[ld]|-{\object@{|}}^{Tb}&
&T^2Y\ar[rr]^(0.35){m_Y}\ar[ld]|-{\object@{|}}_{T^2b}&&TY\ar[ld]|-{\object@{|}}^{Tb}\\
Y\ar@{}[urrr]|{\mbox{\large $\stackrel{\alpha_b}{\Rightarrow}$}}\ar[rr]^{e_Y}&&TY&&T^2Y\ar@{}[urrr]|{\mbox{\large $\stackrel{\beta_b}{\Rightarrow}$}}\ar[rr]^{m_Y}&&TY}\]
the composition of the 2-cells commute, one uses again diagrams (\ref{eq:nat}) and (\ref{eq:coh}). To prove 2-naturality just take in these diagrams a 2-cell $\zeta$ giving a transformation of $(\mT, \V)$-functors instead of $\fii_f$.
\end{proof}

\begin{theorem}
$(T,e,m)$ is a 2-monad on $\VCat$.
\end{theorem}
\begin{proof}
It remains to check the commutativity of the diagrams, for each category $(X,a)$,
\[\xymatrix@!C=7ex{(TX,Ta)\ar[rr]^-{(e_{TX},\alpha_{Ta})}\ar[ddrr]_{(1,1)}&&(T^2X,T^2a)\ar[dd]^{(m_X,\beta_a)}&&(TX,Ta)\ar[ll]_-{(Te_X,\kappa^{-1} T\alpha_a\kappa)}\ar[ddll]^{(1,1)}&(T^3X,T^3a)\ar[rr]^-{(m_{TX},\beta_{Ta})}\ar[dd]_{(Tm_X,\kappa^{-1} T\beta_a\kappa)}&&(T^2X,T^2a)\ar[dd]^{(m_X,\beta_a)}\\
\\
&&(TX,Ta)&&&(T^2X,T^2a)\ar[rr]^-{(m_X,\beta_a)}&&(TX,Ta),}\]
which follows again from diagrams (\ref{eq:nat}) and (\ref{eq:coh}).
\end{proof}

Denoting the 2-category of algebras of this 2-monad by $(\VCat)^\mT$, we get a commutative diagram
\begin{equation}\tag{$\mT$-alg}\label{eq:Talg}
\xymatrix@!R=5ex{\Set^\mT\ar@{}[d]|{\dashv}\ar@<1ex>[d]^{U^\mT}\ar@<-1ex>[rr]\ar@{}[rr]|{\top}&&(\VCat)^\mT\ar@{}[d]|{\dashv}\ar@<1ex>[d]^{U^\mT}\ar@<-1ex>[ll]\\
\Set\ar@<1ex>[u]^{F^\mT}\ar@<-1ex>[rr]\ar@{}[rr]|{\top}&&\VCat.\ar@<-1ex>[ll]\ar@<1ex>[u]^{F^\mT}}
\end{equation}

\section{The fundamental adjunction}\label{sect:FundAdj}
\emph{From now on we assume that $\hat{\beta}_r:Trm_X^\circ\to m_Y^\circ Tr$ is an isomorphism for each $\V$-relation $r:X\relto Y$, so that $m^\circ:T\to T^2$ becomes a pseudo-natural transformation on $\VRel$.}
\bigskip

In this section we will build an adjunction
\begin{equation}\tag{ADJ}\label{eq:ADJ}
(\VCat)^\mT\adjunct{M}{K}\TVCat.
\end{equation}

Let $((Z,c,\eta_c,\mu_c),(h,\fii_h))$ be an object of $(\VCat)^\mT$. The $\V$-category unit $\eta_c$ is a 2-cell
 $1_Z \to c = che_Z$. Let $\widetilde{\mu}_c$ be the 2-cell defined by:
\begin{equation}\tag{$\widetilde{\mu}_c$}\label{eq:muc}
\xymatrix{chT(ch)\ar[r]^-{-\kappa^{-1}_{c,h}}& chTcTh \ar[r]^-{-\fii_h -}& cch Th = cch m_Z\ar[r]^-{\mu_c -}& ch m_Z.}
\end{equation}

\begin{lemma}
The data $(Z, ch,\eta_c,\widetilde{\mu}_c)$ gives a $(\mT, \V)$-category.
\end{lemma}

\begin{proof}
Each of the three  $(\mT, \V)$-category axioms follows from the corresponding  $\V$-category axiom for $(Z,c,\eta_c,\mu_c)$, using (\ref{eq:mon}) and the fact that $(h,\fii_h)$ is an algebra structure.
\end{proof}
\noindent We set
\[K((Z,c,\eta_c,\mu_c),(h,\fii_h))=(Z,ch,\eta_c,\widetilde{\mu}_c).\]
 $K$ extends to a 2-functor in the following way. For a morphism of $\mT$-algebras $(f,\fii_f):((Z,c), h)\to((W,d), k)$, we set $K(f,\fii_f) = (f,\fii_fh)$, where $\fii_f h$ is regarded as a morphism $f ch \longrightarrow dfh=dk Tf$.
For a natural transformation of $\mT$-algebras $\zeta : (f, \fii_f) \to (g, \fii_g)$ we define $K(\zeta)=\zeta h$. By straightforward calculations these indeed define a 2-functor.

Let now $(X,a,\eta_a,\mu_a)$ be a $(\mT,\V)$-category. Let $\ea=Ta m_X^\circ$. Define a 2-cell $\eta_{\ea} : 1_{TX} \to \ea$ by the composite
\begin{equation}\tag{$\eta_{\ea}$}\label{eq:etaa}
\xymatrix{1_{TX} = T1_X \ar[r]^-{T\eta_a}& T(ae_X)\ar[r]^{\kappa^{-1}_{a,e_X}}&Ta Te_X \ar[r]^-{-\lambda_{m_X}-}& Ta m_X^\circ m_XTe_X=Tam_X^\circ,}
\end{equation}
and define $\mu_{\ea} : \ea\ea \to \ea$ by
\[
\xymatrix{
TX \ar[rr]|-{\object@{|}}^{m^\circ_X} \ar[d]|-{\object@{|}}_{m^\circ_X} && T^2X \ar[rr]|-{\object@{|}}^{Ta} \ar[d]|-{\object@{|}}_{m^\circ_{TX}} \ar@{}[rrd]|{\mbox{\large$\stackrel{\hat{\beta}^{-1}}{\Leftarrow}$ }}&& TX  \ar[d]|-{\object@{|}}_{m^\circ_X} \\
T^2X \ar[rr]|-{\object@{|}}^{Tm^\circ_X} \ar@{}[rrd]|{\mbox{\large $\stackrel{\rho_{Tm^\circ_X}}{\Leftarrow}$}} \ar@/_1.5pc/[rrd]_{1_{T^2X}} && T^3X \ar[rr]|-{\object@{|}}^{T^2a} \ar[d]_{Tm_X} \ar@{}[rrd]|{\mbox{\large $\stackrel{-T\mu_a-}{\Leftarrow}$}}&& T^2X \ar[d]|-{\object@{|}}_{Ta} \\
 &&T^2X \ar[rr]|-{\object@{|}}^{Ta} && TX.\\
}
\]

\begin{lemma}
The data $(TX,\ea, \eta_{\ea},\mu_{\ea})$ determines a $\V$-category.
\end{lemma}

\begin{proof}
 The three $\V$-category axioms follow from the corresponding $(\mT, \V)$-category axioms for $(X,a,\eta_a,\mu_a)$.
\end{proof}
Let $\fii_{\ea} : m_XT\ea\to \ea m_X$  be the composite 2-cell
\[\xymatrix@R=1.5em{
T^2X \ar[d]_{m_X} \ar[rr]|-{\object@{|}}^{Tm^\circ_X} \ar@{}[rrd]|{\mbox{\large $\stackrel{}{\Leftarrow}$}} \ar@/^3.pc/[rrrr]|-{\object@{|}}^{T(Tam^\circ_X)}_{\mbox{\large $\stackrel{\kappa^{-1}_{Ta,m^\circ_X}}{\Leftarrow}$}} &&T^3X  \ar[rr]|-{\object@{|}}^{T^2a} \ar[d]_{m_{TX}} \ar@{}[rrd]|{\mbox{\large $\stackrel{\beta_a}{\Leftarrow}$}} &&T^2X \ar[d]^{m_X}\\
TX \ar[rr]|-{\object@{|}}_{m^\circ_X}&&T^2X \ar[rr]|-{\object@{|}}_{Ta} && TX.
}
\]
Wherein the left 2-cell is the mate of the identity map $1_{m_X m_{TX}=m_XTm_X}$. Direct calculations yield:
\begin{lemma}
The pair $(m_X, \fii_{\ea})$ is a $\V$-functor $T(TX, \ea) \to (TX, \ea)$; moreover, it defines a $\mT$-algebra structure on the $\V$-category $(TX, \ea)$.
\end{lemma}
We set
\[M(X,a)=((TX,\ea), (m_X, \fii_{\ea})).\]
We extend this construction to a 2-functor as follows. For a $(\mT,\V)$-functor $(f,\fii_f):(X,a)\to(Y,b)$, $M(f,\fii_f)=(Tf,\widetilde{\fii}_{Tf})$, where $\widetilde{\fii}_{Tf}$ is given by
\[\xymatrix@R=1.5em{
TX \ar[d]_{Tf} \ar[rr]|-{\object@{|}}^{m^\circ_X} \ar@{}[rrd]|{\mbox{\large $\stackrel{\hat{\beta}_f}{\Leftarrow}$}} &&T^2X  \ar[rr]|-{\object@{|}}^{Ta} \ar[d]_{T^2f} \ar@{}[rrd]|{\mbox{\large $\stackrel{-T\fii_f-}{\Leftarrow}$}} &&TX \ar[d]^{Tf}\\
TY \ar[rr]|-{\object@{|}}_{m^\circ_Y}&&T^2Y \ar[rr]|-{\object@{|}}_{Tb} && TY.
}\]
For a natural transformation of $(\mT,\V)$-functors $\zeta : (f, \fii_f) \to (g, \fii_g)$, $M(\zeta)$ is defined by a similar diagram. By direct verification $M$ is a 2-functor.

\begin{theorem}
$M$ is a left 2-adjoint to $K$.
\end{theorem}

\begin{proof}
Given a $(\mT,\V)$-category $(X,a,\eta_a,\mu_a)$,
\[
\xymatrix{(e_X,\widetilde{\alpha}_a):(X,a,\eta_a,\mu_a)\ar[r]& KM(X,a,\eta_a,\mu_a)=(TX,Tam_X^\circ m_X,\eta_{\ea},\widetilde{\mu}_a),}\]
is a $(\mT, \V)$-functor, where $\widetilde{\alpha}_a$ is the composite
\begin{equation}\tag{unit}\label{eq:unit}
\xymatrix{(e_Xa\ar[r]^-{\alpha_a}&Tae_{TX}\ar[rr]^-{-\lambda_{m_X}-}&&Tam_X^\circ m_Xe_{TX}=Tam_X^\circ m_X Te_X),}\end{equation}
These functors define a natural transformation $1 \rightarrow KM$. Given a $\mT$-algebra $((Z,c,\eta_c,\mu_c),(h,\fii_h))$,
\[\xymatrix{(h, \widetilde{\fii}_h): MK((Z,c,\eta_c,\mu_c),(h,\fii_h))=(TZ,T(ch)m_X^\circ,\hat{\eta}_{ch},\mu_{\widehat{ch}}) \ar[r]&  ((Z,c,\eta_c,\mu_c),(h,\fii_h))},\]
is a morphism of $\mT$-algebras, where $\widetilde{\fii}_h$ is defined as
\[hT(ch)m_X^\circ\xrightarrow{-\kappa_{c,h}^{-1}}h TcTh m_X^\circ \xrightarrow{\fii_h -} ch Th m_X^\circ = ch m_Xm_X^\circ \xrightarrow{- \rho_{m_X}} ch,\]
 These define a natural transformation $MK \rightarrow 1$. These natural transformations serve as the unit and the counit of our adjunction. The triangle identities are straightforwardly verified.
\end{proof}

\section{$\mT$ as a $\TVCat$ monad}

Let us identify the 2-monad on $\TVCat$ induced by the adjunction $M \dashv K$, which we denote again by $\mT= (KM = T,e,m)$.

Thus, $T = KM$ is a 2-endofunctor on  $\TVCat$. To a $(\mT, \V)$-category $(X, a,\eta_a,\mu_a)$ it assigns the $(\mT,\V)$-category
$(TX, \ea m_X=Tam_X^\circ m_X,\eta_{\ea},\widetilde{\mu}_{\ea})$ with components defined in the diagrams (\ref{eq:etaa}) and (\ref{eq:muc}) of the last section, to a
$(\mT, \V)$-functor $(f, \fii_f)$ it assigns the $(\mT,\V)$-functor $(Tf,\widetilde{\fii}_f)$ which can be diagrammatically specified by
\[\xymatrix{T^2X\ar[d]_{m_X}\ar[rr]^{T^2f}&&T^2Y\ar[d]^{m_Y}\\
TX\ar[d]|-{\object@{|}}_{m_X^\circ}\ar[rr]^{Tf}&&TY\ar[d]|-{\object@{|}}^{m_Y^\circ}\\
T^2X\ar@{}[urr]|{\mbox{\large $\stackrel{\ebeta_f}{\Rightarrow}$}}\ar[d]|-{\object@{|}}_{Ta}\ar[rr]^{T^2f}&&T^2Y\ar[d]|-{\object@{|}}^{Tb}\\
TX\ar@{}[urr]|{\mbox{\large $\stackrel{-T\fii_f-}{\Rightarrow}$}}\ar[rr]^{Tf}&&TY,}\]
 and the $\mT$-image of a $(\mT, \V)$-natural transformation $\zeta : (f, \fii_f) \to (g, \fii_g)$ is computed by a similar diagram.

The unit of the 2-monad is the unit $(e,\widetilde{\alpha})$ of the adjunction $K \dashv M$ defined in (\ref{eq:unit}). The multiplication of the 2-monad is given by $(m,\widetilde{\beta})$, the component of which at a $(\mT, \V)$-category $(X, a)$,
 -- which is a $(\mT, \V)$-functor $MKMK(X, a) \to MK(X, a)$ --, is  pictorially described by:
\[\xymatrix{T^3X\ar `l^d[ddddd]`^r[ddddd]|-{\object@{|}}_{MKMK(a)} [ddddd]
\ar[d]_{m_{TX}}\ar[rr]^{Tm_X}&&T^2X\ar `r_d[ddddd]`_l[ddddd]|-{\object@{|}}^{MK(a)} [ddddd]\ar[d]^{m_X}\\
T^2X\ar[rr]^{m_X}\ar[d]|-{\object@{|}}_{m_{TX}^\circ}&&TX\ar[d]^1\\
T^3X\ar[rr]_{m_Xm_{TX}}\ar@{}[urr]|{\mbox{\large $\stackrel{-\rho_{m_{TX}}}{\Rightarrow}$}}\ar[d]_{Tm_X}&&TX\ar[d]^1\\
T^2X\ar[d]|-{\object@{|}}_{Tm_X^\circ}\ar[rr]^{m_X}&&TX\ar[d]|-{\object@{|}}^{m_X^\circ}\\
T^3X\ar[rr]^{m_{TX}}\ar@{}[rru]|{\mbox{\large $\stackrel{(-\rho_{Tm_X})(\lambda_{m_X}-)}{\Rightarrow}$}}\ar[d]|-{\object@{|}}_{T^2a}&&T^2X\ar[d]|-{\object@{|}}^{Ta}\\
T^2X\ar[rr]^-{m_X}\ar@{}[rru]|{\mbox{\large $\stackrel{\beta_{a}}{\Rightarrow}$}}&&TX.}\]

\begin{theorem}
The 2-monad $(T, e,m)$ on $\TVCat$ is a KZ monad.
\end{theorem}

\begin{proof}
One of the equivalent conditions expressing the KZ property is the existence of a modification $\delta : Te \to eT : T \to TT$ such that
\begin{equation}\tag{mod}\label{eq:mod}
\delta e = 1_{ee}\mbox{ and }m\delta = 1_{1_T}.
\end{equation}

For a $(\mT, \V)$-category $(X, a, \mu_a, \eta_a)$, let $\delta_{(X, a)}$ be the composite 2-cell
\[\xymatrix{e_{TX} \ar[r]^-{T^2\eta_a-}& T^2(ae_X)e_{TX} \ar[r]^{T\kappa_{a,e_X}-} &T(TaTe_X)e_{TX} \ar[r]^{\kappa_{Ta,Te_X}-} & T^2aT^2e_Xe_{TX}}\]
\[\xymatrix{= T(Ta)e_{T^2X}Te_X\ar[rrr]^-{T(Ta\lambda_{m_X})\lambda_{m_{TX}}-}&&& T(Tam_X^\circ m_X)m_{TX}^\circ m_{TX}e_{T^2X}Te_X.}\]
This defines a  $(\mT, \V)$-natural transformation
\[\delta_{(X,a)}:(Te_X, T\widetilde{\alpha}_a) \to (e_{TX}, \widetilde{\alpha}_{\ea m_X}).\]
The family of these natural transformations gives the required modification $Te \to eT$. The first of the two required equalities (\ref{eq:mod}) is straightforward. The second one follows from (mon).
\end{proof}

\section{Representable $(\mT,\V)$-categories: from Nachbin spaces\\to Hermida's representable multicategories}

Being a KZ monad, for the monad $\mT$ on $\TVCat$ a $\mT$-algebra structure on a $(\mT,\V)$-category $(X,a)$ is, up to isomorphism, a reflective left adjoint to the unit $e_{(X,a)}$; hence, having a $\mT$-algebra structure is a property, rather than an additional structure, for any $(\mT,\V)$-category. As Hermida in \cite{Her00}, we say that:
\begin{definition}
A $(\mT,\V)$-category is \emph{representable} if it has a pseudo-algebra structure for $\mT$.
\end{definition}

In the diagram below $(\TVCat)^\mT$ is the 2-category of $\mT$-algebras, $F^\mT\dashv G^\mT$ is the corresponding adjunction, and $\widetilde{K}$ is the comparison 2-functor:
\[\xymatrix{(\VCat)^\mT\ar@{}[rd]|{\top}\ar[rr]^{\widetilde{K}}\ar@<1ex>[rd]^K&&(\TVCat)^\mT.\ar@<1ex>[ld]^{G^\mT}\ar@{}[ld]|{\bot}\\
&\TVCat\ar@<1ex>[lu]^M\ar@<1ex>[ru]^{F^\mT}}\]
The composition of the adjunctions $F^\mT\dashv G^\mT$ and $A^\circ\dashv A_e$ (see (\ref{eq:adj}) in Section \ref{sect:three}) gives an adjunction $F_e^\mT\dashv G_e^\mT$ that induces again the monad $\mT$ on $\VCat$. Let $\widetilde{A}_e$ be the corresponding comparison 2-functor as depicted in the following diagram:
\[\xymatrix{(\VCat)^\mT\ar@<-1ex>@{}[rddd]|{\top}\ar@{.>}[rddd]\ar@{}[rd]|{\top}\ar@<-1ex>[rr]_{\widetilde{K}}\ar@<1ex>[rd]^K&&(\TVCat)^\mT.
\ar@<1ex>[ld]^{G^\mT}\ar@{}[ld]|{\bot}\ar@{.>}@<-1ex>[ll]_{\widetilde{A}_e}\ar@{.>}@<2ex>[lddd]^{G_e^\mT}\ar@{}@<1ex>[lddd]|{\bot}\\
&\TVCat\ar@<1ex>[lu]^M\ar@<1ex>[ru]^{F^\mT}\ar@{.>}@<1ex>[dd]^{A_e}\ar@{}[dd]|{\dashv}\\
\\
&\VCat\ar@{.>}@<2ex>[luuu]\ar@{.>}[ruuu]^{F_e^\mT}\ar@<1ex>@{.>}[uu]^{A^\circ}}\]

\begin{theorem}\label{th:rep}
$\widetilde{K}$ and  $ \widetilde{A}_e$ define an adjoint 2-equivalence.
\end{theorem}

\begin{proof}
The isomorphism  $\widetilde{A}_e\widetilde{K} \cong 1$ can be directly verified. We will establish that  $\widetilde{K}\widetilde{A}_e \cong 1$.

Suppose that a $(\mT, \V)$-functor $(f, \fii_f) : T(X, a) \rightarrow (X, a)$ is a $\mT$-algebra structure on a $(\mT, \V)$-category $(X, a)$. Observe that the underlying $\V$-relation of the representable $(\mT, \V)$-category $\widetilde{K}\widetilde{A}_e((X, a), (f, \fii_f))$ is $ae_Xf : TX \rel TX$.

Since $\mT$ is a KZ monad, following \cite{KellyLack},  $(f, \fii_f)$ is a left adjoint to the unit $(e_X, \widetilde{\alpha}_a)$ of $\mT$. By Proposition \ref{th:adj} we get an isomorphism
\[\omega : e^\circ_XTam^\circ_Xm_X \rightarrow aTf.\]
Let $\iota$ denote the composite isomorphism
\[ae_Xf = aTfe_{T_X} \xrightarrow{\omega^{-1}-} e^\circ_XTam^\circ_Xm_Xe_{T_X} =e^\circ_XTam^\circ_Xm_XTe_{X} \xrightarrow{\omega-} aTfTe_{X} = a.\]
It can be verified that the pair $(1_X, \iota)$ is an isomorphism $\widetilde{K}\widetilde{A}_e((X, a), (f, \fii_f)) \to ((X, a), (f, \fii_f))$ in $(\TVCat)^{\mT}$. The family of these morphisms determine the required 2-natural isomorphism  $\widetilde{K}\widetilde{A}_e \cong 1$.
\end{proof}

We explain now how representable $(\mT,\V)$-categories capture two important cases which were developed independently.

\subsection*{Nachbin's ordered compact Hausdorff spaces.}
For $\V = 2$ and $\mT=\mU=(U,e,m)$ the ultrafilter monad extended to $\two$-$\Rel=\Rel$ as in \cite{Ba}, so that, for any relation $r:X\relto Y$, $Ur=Uq(Up)^\circ$, where $p:R\to X$, $q:R\to Y$ are the projections of $R=\{(x,y)\mid x\,r\,y\}$. Then $\Cats{2}\simeq\Ord$ and the functor $U:\Ord\to\Ord$ sends an ordered set $(X,\le)$ to $(UX,U\!\le)$ where
\[
 \xx\,(U\!\le)\,\yy\hspace{1em}\text{whenever}\hspace{1em}\forall A\in\xx,B\in\yy\exists x\in A,y\in B\,.\,x\le y,
\]
for all $\xx,\yy\in UX$. The algebras for the monad $\mU$ on $\Ord$ are precisely the ordered compact Hausdorff spaces as introduced in \cite{Nac50}:

\begin{definition}
An \emph{ordered compact Hausdorff space} is an ordered set $X$ equipped with a compact Hausdorff topology so that the graph of the order relation is a closed subset of the product space $X\times X$.
\end{definition}

We denote the category of ordered compact Hausdorff spaces and monotone and continuous maps by $\ORDCH$. It is shown in \cite{Tho09} that, for a compact Hausdorff space $X$ with ultrafilter convergence $\alpha:UX\to X$ and an order relation $\le$ on $X$, the set $\{(x,y)\mid x\le y\}$ is closed in $X\times X$ if and only if $\alpha:UX\to X$ is monotone; and this shows
\[
 \ORDCH\simeq\Ord^\mU,
\]
and the diagram (\ref{eq:Talg}) at the end of Section \ref{sect:VCatMonad} becomes
\[
\xymatrix@!R=5ex{\CompHaus\ar@{}[d]|{\dashv}\ar@<1ex>[d]^{U^\mT}\ar@<-1ex>[rr]\ar@{}[rr]|{\top}&&\ORDCH\ar@{}[d]|{\dashv}\ar@<1ex>[d]^{U^\mT}\ar@<-1ex>[ll]\\
\Set\ar@<1ex>[u]^{F^\mT}\ar@<-1ex>[rr]\ar@{}[rr]|{\top}&&\Ord.\ar@<-1ex>[ll]\ar@<1ex>[u]^{F^\mT}}
\]
The functor $K:\ORDCH\to\TOP=(\mU,\two)$-$\Cat$ of Section \ref{sect:FundAdj} can now be described as sending $((X,\leq),\alpha:UX\to X)$ to the space $KX=(X,a)$ with ultrafilter convergence $a:UX\relto X$ given by the composite
\[\xymatrix{UX\ar[r]^-{\alpha}&X\ar[r]|-{\object@{|}}^-{\leq}&X;}\]
of the order relation $\le:X\relto X$ of $X$ with the ultrafilter convergence $\alpha:UX\to X$ of the compact Hausdorff topology of $X$. In terms of open subsets, the topology of $KX$ is given precisely by those open subsets of the compact Hausdorff topology of $X$ which are down-closed with respect to the order relation of $X$. On the other hand, for a topological space $(X,a)$, the ordered compact Hausdorff space $MX$ is the set $UX$ of all ultrafilters of $X$ with the order relation
\[\xymatrix{UX\ar[r]|-{\object@{|}}^-{m_X^\circ}&UUX\ar[r]|-{\object@{|}}^-{Ua}&UX,}\]
and with the compact Hausdorff topology given by the convergence $m_X:UUX\to UX$; put differently, the order relation on $UX$ is defined by
\[
 \xx\le\yy\iff\forall A\in\xx\,.\,\overline{A}\in\yy,
\]
and the compact Hausdorff topology on $UX$ is generated by the sets
\[
 \{\xx\in UX\mid A\in\xx\}\hspace{2em}(A\subseteq X).
\]
The monad $\mU=(U,e,m)$ on $\Top$ induced by the adjunction $M\dashv K$ assigns to each topological space $X$ the space $UX$ with basic open sets
\[
 \{\xx\in UX\mid A\in\xx\}\hspace{2em}(A\subseteq X\text{ open}).
\]
By definition, a topological space $X$ is called \emph{representable} if $X$ is a pseudo-algebra for $\mU$, that is, whenever $e_X:X\to UX$ has a (reflective) left adjoint. Note that a left adjoint of $e_X:X\to UX$ picks, for every ultrafilter $\xx$ on $X$, a smallest convergence point of $\xx$. The following result provides a characterisation of representable topological spaces.

\begin{theorem}
Let $X$ be a topological space. The following assertions are equivalent.
\begin{enumerate}[\em (i)]
\item $X$ is representable.
\item $X$ is locally compact and every ultrafilter has a smallest convergence point.
\item $X$ is locally compact, weakly sober and the way-below relation on the lattice of open subsets is stable under finite intersection.
\item $X$ is locally compact, weakly sober and finite intersections of compact down-sets are compact.
\end{enumerate}
\end{theorem}

Representable T$_0$-spaces are known under the designation \emph{stably compact spaces}, and are extensively studied in \cite{GHK+03,Jun04,Law11} and  \cite{Sim82} (called \emph{well-compact spaces} there). One can also find there the following characterisation of morphisms between representable spaces.

\begin{theorem}
Let $f:X\to Y$ be a continuous map between representable spaces. Then the following are equivalent.
\begin{enumerate}[\em (i)]
\item $f$ is a pseudo-homomorphism.
\item For every compact down-set $K\subseteq Y$, $f^{-1}(K)$ is compact.
\item The frame homomorphism $f^{-1}:\calO Y\to\calO X$ preserves the way-below relation.
\end{enumerate}
\end{theorem}

\subsection*{Hermida's representable multicategories}

We sketch now some of the main achievements of \cite{Her00,Her01} which fit in our setting and can be seen as counterparts to the classical topological results mentioned above. In \cite{Her00,Her01} Hermida is working in a finitely complete category $\catB$ admitting free monoids so that the free-monoid monad $\mM=\mmonad$ is Cartesian; however, for the sake of simplicity we consider only the case $\catB=\Set$ here. We write $\SPAN$ to denote the bicategory of spans in $\Set$, and recall that a \emph{category} can be viewed as a span
\[
 \xymatrix{& C_1\ar[dl]_d\ar[dr]^c\\ C_0 && C_0}
\]
which carries the structure of a monoid in the category $\SPAN(C_0,C_0)$. The 2-category of monoids in $\Cat$ (aka strict monoidal categories) and strict monoidal functors is denoted by $\MONCAT$, and the diagram (\ref{eq:Talg})
becomes

\[\xymatrix@!R=5ex{\Mon\ar@{}[d]|{\dashv}\ar@<1ex>[d]^{U^\mT}\ar@<-1ex>[rr]\ar@{}[rr]|{\top}&&\MONCAT\ar@{}[d]|{\dashv}\ar@<1ex>[d]^{U^\mT}\ar@<-1ex>[ll]\\
\Set\ar@<1ex>[u]^{F^\mT}\ar@<-1ex>[rr]\ar@{}[rr]|{\top}&&\Cat.\ar@<-1ex>[ll]\ar@<1ex>[u]^{F^\mT}}\]

A \emph{multicategory} can be viewed as a span
\[
 \xymatrix{& C_1\ar[dl]_d\ar[dr]^c\\ MC_0 && C_0}
\]
in $\Set$ together with a monoid structure in an appropriate category. This amounts to the following data:
\begin{enumerate}[--]
\item a set $C_0$ of objects;
\item a set $C_1$ of arrows where the domain of an arrow $f$ is a sequence $(X_1,X_2,\ldots,X_n)$ of objects and the codomain is an object $X$, depicted as
\[
 f:(X_1,X_2,\ldots,X_n)\to X;
\]
\item an identity $1_X:(X)\to X$;
\item a composition operation.
\end{enumerate}
The 2-category of multicategories, morphisms of multicategories and appropriate 2-cells is denoted by $\MultiCat$. Keeping in mind that $\SPAN$ is equivalent to $\Rels{\Set}$, for $\V=\Set$ and $\mT=\mM$, the fundamental adjunction (\ref{eq:ADJ}) of Section \ref{sect:FundAdj} specialises to:

\begin{theorem}
There is a 2-monadic 2-adjunction $\MultiCat\adjunct{M}{K}\MONCAT$.
\end{theorem}

Here, for a strict monoidal category
\[
 \xymatrix{& C_1\ar[dl]_d\ar[dr]^c\\ C_0 && C_0}
\]
with monoid structure $\alpha:MC_0\to C_0$ on $C_0$, the corresponding multicategory is given by the composite of
\[
 \xymatrix{& MC_0\ar[dl]_1\ar[dr]^\alpha && C_1\ar[dl]_d\ar[dr]^c\\ MC_0 && C_0 && C_0}
\]
in $\SPAN$; and to a multicategory
\[
 \xymatrix{& C_1\ar[dl]_d\ar[dr]^c\\ MC_0 && C_0}
\]
one assigns the strict monoidal category
\[
 \xymatrix{&& MC_1\ar[dl]_d\ar[dr]^c\\ & MMC_0\ar[dl]_{m_{C_0}} && MC_0\\ MC_0}
\]
where the objects in the span are free monoids.

The induced 2-monad on $\MultiCat$ is of Kock-Z\"oberlein type, and a \emph{representable multicategory} is a pseudo-algebra for this monad. In elementary terms, a multicategory
\[
 \xymatrix{& C_1\ar[dl]_d\ar[dr]^c\\ MC_0 && C_0}
\]
is representable precisely if for every $(x_1,\ldots,x_n)\in MC_0$ there exists a morphism (called universal arrow)
\[
 (x_1,\ldots,x_n)\to \otimes(x_1,\ldots,x_n)
\]
which induces a bijection
\[
 \hom((x_1,\ldots,x_n),y)\simeq\hom(\otimes(x_1,\ldots,x_n),y),
\]
natural in $y$, and universal arrows are closed under composition.

\section{Duals for $(\mT,\V)$-categories}

For a $\V$-category $(Z, c) = (Z, c, \eta_c, \mu_c)$, the \emph{dual} $D(Z, c)$ of $(Z, c)$ is defined to be the $\V$-category $Z^\op=(Z, c^\op, \eta_{c^\op}, \mu_{c^\op})$, with $c^\op=c^\circ$, $\eta_{c^\op}=\eta_c^\circ$ and $\mu_{c^\op}=\mu_c^\circ$. This construction extends to a 2-functor
\[D : \VCat \to \VCat^\co\]
as follows. For a $\V$-functor $(f, \fii_f) : (Z, c) \to (W, d)$ set $D(f, \fii_f) = f^\op=(f, \fii_f^\op):(Z,c^\circ)\to(W,d^\circ)$, where $\fii_f^\op$ is defined by
\[\xymatrix{f c^\circ \ar[r]^-{-\lambda_f}& f c^\circ f^\circ f = f (fc)^\circ f \ar[r]^-{-(\fii_f)^\circ -}& f (df)^\circ f = ff^\circ d^\circ f \ar[r]^-{\rho_f -}& d^\circ f.}\]
On 2-cells $\zeta: (f, \fii_f) \to (g, \fii_g)$ of $\VCat$, set $D(\zeta) = \zeta^\op$, which is defined analogously by
\[\xymatrix{f c^\circ \ar[r]^-{-\lambda_g}& f c^\circ g^\circ g = f (gc)^\circ g \ar[r]^-{-\zeta^\circ -}& f (df)^\circ g = ff^\circ d^\circ g \ar[r]^-{\rho_f -}& d^\circ g.}\]

The monad $\mT$ on $\VCat$ of Section \ref{sect:VCatMonad} gives rise to a monad $\mT$ on $\VCat^\co$.
From now on \emph{we assume that $T(c^\circ)=(Tc)^\circ$ for every $\V$-relation $c$}.
Let  $((Z, c), (h, \fii_h))$ be a $\mT$-algebra. Then
\[\xymatrix{(TZ, Tc^\circ) \ar[rr]^{D(h, \fii_h)} && (Z, c^\circ)}\]
gives a $\mT$-algebra structure on $(Z, c^\circ)$, which we write as $((Z, c^\circ), h^\op)$.

\begin{definition}
The \emph{dual} of a $\mT$-algebra $((Z, c), h)$ is the  $\mT$-algebra $(Z^\op,h^\op)=((Z, c^\circ), h^\op)$.
\end{definition}
This construction extends to a 2-functor
\begin{equation}\tag{Dual}\label{eq:Dual}
D : (\VCat)^\mT \longrightarrow ((\VCat)^\mT)^\co
\end{equation}
as follows. If $(f, \fii_f): ((Z, c), h) \to ((W,d), k)$ is a morphism of $\mT$-algebras, then $D(f, \fii_f)=f^\op:((Z, c^\circ), h^\op) \to ((W, d^\circ), k^\op)$ is a morphism of $\mT$-algebras, and if $\zeta:(f,\fii_f)\to(g,\fii_g)$ is a 2-cell in $(\VCat)^\mT$, then $D(\zeta)=\zeta^\op:D(g, \fii_g) \to D(f, \fii_f)$ is a 2-cell in $\VCat^\mT$.

Using the adjunction $M\dashv K$ we can define the dual of a $(\mT,\V)$-category using the construction of duals in $(\VCat)^\mT$ via the composition:
\[\xymatrix@=8ex{(\VCat)^\mT\ar@`{(-20,-20),(-20,20)}^D\ar@{}[r]|{\top}\ar@<1mm>@/^2mm/[r]^{{K}} & \TVCat.\ar@<1mm>@/^2mm/[l]^{{M}}}\]

\begin{definition}
The \emph{dual} of a $(\mT, \V)$-category $(X,a)$ is the $(\mT,\V)$-category $KDM(X,a)$; that is,
\[X^\op=(TX,m_XTa^\circ m_X).\]
\end{definition}

For representable $(\mT,\V)$-categories $(X,a)$ we can use directly extensions of $\widetilde{K}$ and $\widetilde{A}_e$ to pseudo-algebras, so that we can obtain a dual structure $X^{\widetilde{\op}}$ on the same underlying set $X$ via the composition $\widetilde{K}D\widetilde{A}_e$:
\[\xymatrix@=8ex{(\VCat)^\mT\ar@`{(-20,-20),(-20,20)}^D\ar@{}[r]|-{\top}\ar@<1mm>@/^2mm/[r]^-{\widetilde{K}} & (\TVCat)^\mT.\ar@<1mm>@/^2mm/[l]^-{\widetilde{A}_e}}\]
Then it is easily checked that, for any $(\mT,\V)$-category $X$,
\[X^\op=(TX)^{\widetilde{\op}},\]
since $TX$, as a free $\mT$-algebra on $\TVCat$, is representable.

For $\V$ a quantale, duals of $(\mT,\V)$-categories proved to be useful in the study of (co)completeness (see \cite{CH09,CH09a,Hof11}). Next we outline briefly the setting used and the role duals play there.

Let $\V$ be a quantale. When the lax extension of $T:\Set\to\Set$ to $\VRel$ is determined by a map $\xi:TV\to V$ which is a $\mT$-algebra structure on $\V$ (for the $\Set$-monad $\mT$) as outlined in \cite[Section 4.1]{CH09}, then, under suitable conditions, $\V$ itself has a natural $(\mT,\V)$-category structure $\hom_\xi$ given by the composite
\begin{equation}\tag{$(\mT,\V)$-$\hom$}\label{eq:TVhom}
\xymatrix{TV\ar[r]^\xi&V\ar[r]|-{\object@{|}}^\hom&V,}
\end{equation}
where $\hom$ is the internal hom on $V$.\footnote{This is the case when a \emph{topological theory} in the sense of \cite{Hof07} is given; see \cite{Hof07} for details.} Then the well-known equivalence:\\
\begin{quotation}
  Given $\V$-categories $(X,a)$, $(Y,b)$, for a $\V$-relation $r:X\relto Y$,\vspace*{2mm}
\begin{itemize}
\item[] $r:(X,a)\relto (Y,b)$ is a $\V$-module (or profunctor, or distributor)
\item[] $\iff$ the map $r:X^\op\otimes (Y,b)\to(\V,\hom)$ is a $\V$-functor.
\end{itemize}
\end{quotation}
can be generalized to the $(\mT,\V)$-setting. Here a \emph{$(\mT,\V)$-relation} $r:X\krelto Y$ is a $\V$-relation $TX\relto Y$, and $(\mT,\V)$-relations $X\stackrel{r}{\krelto} Y\stackrel{s}{\krelto} Z$ compose as $\V$-relations as follows:
\[\xymatrix{TX\ar[r]|-{\object@{|}}^{m_X^\circ}&T^2X\ar[r]|-{\object@{|}}^{Tr}&TY\ar[r]|-{\object@{|}}^s&Z;}\]
we denote this composition by $s\circ r$. A \emph{$(\mT,\V)$-module} $\fii:(X,a)\krelto(Y,b)$ between $(\mT,\V)$-categories $(X,a)$, $(Y,b)$ is a $(\mT,\V)$-relation such that
\[\fii\circ a=\fii=b\circ\fii.\]
The next result can be found in \cite{CH09} (see also \cite[Remark 5.1 and Lemma 5.2]{Hof13}).

\begin{theorem}\label{prop:Mod_vs_Fun}
Let $(X,a)$ and $(Y,b)$ be $(\mT,\V)$-categories and $\varphi:X\krelto Y$ be a $(\mT,\V)$-relation. The following assertions are equivalent.
\begin{enumerate}[\em (i)]
\item $\varphi:(X,a)\krelto (Y,b)$ is a $(\mT,\V)$-module.
\item The map $\varphi:TX\times Y\to\V$ is a $(\mT,\V)$-functor $\varphi:X^\op\otimes (Y,b)\to(\V,\hom_\xi)$.
\end{enumerate}
\end{theorem}

In particular, the $(\mT,\V)$-relation $a:X\krelto X$ is a $(\mT,\V)$-module from $(X,a)$ to $(X,a)$. Although $\TVCat$ is in general not monoidal closed for $\otimes$, the functor $X^\op\otimes-:\TVCat\to\TVCat$ has a right adjoint $(-)^{X^\op}:\TVCat\to\TVCat$ for every $(\mT,\V)$-category $X$, and from the $(\mT,\V)$-module $a$ we obtain the \emph{Yoneda $(\mT,\V)$-functor}
\[
 y_X:X\to\V^{X^\op}.
\]
By Theorem \ref{prop:Mod_vs_Fun}, we can think of the elements of $\V^{X^\op}$ as $(\mT,\V)$-modules from $(X,a)$ to $(1,e_1^\circ)$. The following result was proven in \cite{CH09} and provides a Yoneda-type Lemma for $(\mT,\V)$-categories.

\begin{theorem}\label{lem:Yoneda}
Let $(X,a)$ be a $(\mT,\V)$-category. Then, for all $\psi$ in $\V^{X^\op}$ and all $\xx\in TX$,
\[
 \llbracket Ty_X(\xx),\psi\rrbracket=\psi(\xx),
\]
with $\llbracket-,-\rrbracket$ the $(\mT,\V)$-categorical structure on $\V^{X^\op}$.
\end{theorem}

To generalise these results to the general setting studied in this paper, that is when $\V$ is not necessarily a thin category, one faces a first obstacle: When can we equip the category $\V$ with a canonical (although non-legitimate) $(\mT,\V)$-category structure as in (\ref{eq:TVhom})? The obstacle seems removable when $\mT=\mM$ is the free-monoid monad.
In fact, as above, the monoidal structure $(X_1,\ldots,X_n)\mapsto X_1\otimes\dots\otimes X_n$ defines a lax extension of $\mM$ to $\Rels{\V}$, a monoidal structure on $\Cats{(\mM,\V)}\simeq\V$-$\MultiCat$, and it turns $\V$ into a generalised multicategory. We therefore conjecture that Theorems \ref{prop:Mod_vs_Fun} and \ref{lem:Yoneda} hold also in this more general situation; however, so far we were not able to prove this.

\section*{Acknowledgments}

The first author acknowledges partial financial  
assistance by the Centro de Matem\'{a}tica da Universidade de Coimbra (CMUC),
funded by the European Regional Development Fund through the program COMPETE

and by the Portuguese Government through FCT, under the project PEst-C/MAT/UI0324/2013 and grant number SFRH/BPD/79360/2011, and by the Presidential Grant for Young Scientists, PG/45/5-113/12, of the National Science Foundation of Georgia.

The second author acknowledges partial financial  
assistance by CMUC,
funded by the European Regional Development Fund through the program COMPETE
and by the Portuguese Government through FCT, under the project PEst-C/MAT/UI0324/2013.

The third author acknowledges partial financial assistance by Portuguese funds through CIDMA (Center for Research and Development in Mathematics and Applications), and the Portuguese Foundation for Science and Technology (``FCT -- Funda\c{c}\~ao para a Ci\^encia e a Tecnologia''), within the project PEst-OE/MAT/UI4106/2014, and by the project NASONI under the contract PTDC/EEI-CTP/2341/2012.

 \end{document}